%% file: 3braidspaper.tex
\documentclass{amsart}[12pt]
\usepackage{graphicx,url,amsthm,xspace,enumitem,subfigure,amssymb,mathtools,mathrsfs}
\usepackage[dvipsnames]{xcolor}
\usepackage{amscd, amsmath}
\usepackage{multirow, overpic}
\usepackage{tikz}
\usepackage{booktabs}
\usepackage{xy}
\usepackage{tikz-cd}

\usetikzlibrary{cd}

\usepackage{pinlabel}
\usepackage[linktocpage]{hyperref}

\usepackage[boxruled]{algorithm2e}

\allowdisplaybreaks

\hypersetup{
  colorlinks,
  citecolor=Periwinkle,
  linkcolor=Black,
  urlcolor=PineGreen}

\usepackage[margin=3cm, marginpar=2.5cm]{geometry}

\theoremstyle{plain}
\newtheorem{thm}{Theorem}[section]

\newtheorem{lemma}[thm]{Lemma}
\newtheorem{prop}[thm]{Proposition}
\newtheorem{cor}[thm]{Corollary}

\theoremstyle{definition}
\newtheorem{quest}{Question}[section]
\newtheorem{dfn}[quest]{Definition}

\theoremstyle{remark}
\newtheorem{remark}[thm]{Remark}
\newtheorem{eg}[thm]{Example}

\theoremstyle{plain}

\numberwithin{equation}{section}

\newcommand{\QQ}{\mathbb{Q}}
\newcommand{\RR}{\mathbb{R}}

\newcommand{\ZZ}{\mathbb{Z}}
\newcommand{\cptwo}{\mathbb{CP}^2}
\newcommand{\defeq}{\vcentcolon=}

\makeatletter
\DeclareRobustCommand\widecheck[1]{{\mathpalette\@widecheck{#1}}}
\def\@widecheck#1#2{%
	\setbox\z@\hbox{\m@th$#1#2$}%
	\setbox\tw@\hbox{\m@th$#1%
		\widehat{%
			\vrule\@width\z@\@height\ht\z@
			\vrule\@height\z@\@width\wd\z@}$}%
	\dp\tw@-\ht\z@
	\@tempdima\ht\z@ \advance\@tempdima2\ht\tw@ \divide\@tempdima\thr@@
	\setbox\tw@\hbox{%
		\raise\@tempdima\hbox{\scalebox{1}[-1]{\lower\@tempdima\box
				\tw@}}}%
	{\ooalign{\box\tw@ \cr \box\z@}}}
\makeatother

\title{Obstructing Lagrangian concordance for closures of 3-braids}

\author[Wu]{Angela Wu}
\address[A.\ Wu]{Department of Mathematics \\ Louisiana State University \\ Baton Rouge \\ LA \\ U.S.A.}
\email{awu@lsu.edu}

\begin{document}
	
\maketitle

\begin{abstract}
 We show that any knot which is smoothly the closure of a 3-braid cannot be Lagrangian concordant to and from the maximum Thurston-Bennequin Legendrian unknot except the unknot itself. Our obstruction comes from drawing the Weinstein handlebody diagrams of particular symplectic fillings of cyclic branched double covers of knots in $S^3$. We use the Legendrian contact homology differential graded algebra of the links in these diagrams to compute the symplectic homology of these fillings to derive a contradiction. As a corollary, we find an infinite family of contact manifolds which are rational homology spheres but do not embed in $\RR^4$ as contact type hypersurfaces.
\end{abstract}

\input{Introduction.tex}

\input{Chapter3.tex}
\input{Chapter4.tex}
\input{Chapter5.tex}

\bibliography{references}
\bibliographystyle{amsalpha}
\end{document}

%% file: Introduction.tex
\section{Introduction}
\label{intro}

Knot concordance was first defined by Fox and Milnor in \cite{FoxMil66} as a way to endow topological knots with a group structure. Two knots are said to be smoothly concordant if they jointly form the boundary of a smooth cylinder in four-dimensional Euclidean space. We study a variant of the problem of concordance defined in the symplectic setting by Chantraine \cite{Cha10} called \emph{Lagrangian concordance}. 

We consider Legendrian knots in $\RR^3$ with the standard contact structure $\ker\alpha$, $\alpha = dz-ydx$. Let $\RR^4$ be the symplectization of $\RR^3$, $\RR_t\times \RR^3$ with the symplectic form $\omega=d(e^t\alpha)$. 

\begin{dfn}
	Let $\Lambda_+,\Lambda_-\subset \RR^3$ be Legendrian knots where $\RR^3$ is equipped with the standard contact structure $\xi=\ker(\alpha)$. A \textit{Lagrangian cobordism from $\Lambda_-$ to $\Lambda_+$} is a Lagrangian $L$ embedded in $\RR^4$ such that
	$ ((-\infty,-T)\times \RR^3) \cap L = (-\infty,-T)\times \Lambda_-$
	and
	$((T,\infty)\times\RR^3) \cap L = (T,\infty)\times\Lambda_+$
	for some $T\geq 0$.
\end{dfn}

\begin{dfn}	
	If a Lagrangian cobordism has zero genus, then we call it a \textit{Lagrangian concordance}. If there is a Lagrangian concordance from $\Lambda_-$ to $\Lambda_+$, we write $\Lambda_-\prec\Lambda_+$ and say that $\Lambda_-$ is \textit{Lagrangian concordant to} $\Lambda_+$.
\end{dfn}

Lagrangian cobordisms between Legendrian submanifolds have been studied in symplectic and contact topology due to their key role in symplectic field theory \cite{EliGiv00}. Indeed, Lagrangian cobordisms were first defined to construct a category whose objects are Legendrian submanifolds and whose morphisms are given by the exact Lagrangian cobordisms. This led to the development of a new invariant by Chekanov and Eliashberg \cite{Che02,Eli98}, a differential graded algebra 
whose homology, called \emph{Legendrian contact homology} gives a functor from the category of Legendrians to the category of these differential graded algebras. The study of Lagrangian cobordisms has also led to the development of the wrapped Fukaya category \cite{FukSei08}. 
With this motivation, a lot of recent work aims to understand the structure and behaviour of the Lagrangian cobordism relation (for instance \cite{BlaLeg20}, \cite{Cha10}, \cite{Cha13}, \cite{CorNg16}, \cite{EkhHon16}, \cite{SabTra13}, \cite{SabVel21}).

In \cite{Cha10}, Chantraine showed that Legendrian isotopic Legendrian knots are Lagrangian concordant, and that Lagrangian concordance can be obstructed by classical Legendrian knot invariants: the rotation number and the Thurston-Bennequin number. Lagrangian concordance is both reflexive and transitive, suggesting that as a relation, it is potentially a partial order on Legendrian knots. In \cite{Cha13}, Chantraine proved that Lagrangian concordance is not symmetric. It is not known if Lagrangian concordance is antisymmetric, ie. if $\Lambda_1\prec\Lambda_2\prec\Lambda_1$, then $\Lambda_1$ is Legendrian isotopic to $\Lambda_2$. Let $U$ denote the $tb=-1$ Legendrian unknot. As a first step to understanding the potential antisymmetry of this relation, we may pose the following question about the simplest case:

\begin{quest}\label{mainquestion} 
Which Legendrian knots $\Lambda$ satisfy $U\prec\Lambda\prec U$ (as in Figure \ref{fig:ulambdau})?
\end{quest}

The answer to this question is not known. In fact, it is not known whether any knot is Lagrangian concordant to $U$, other than $U$ itself. If a Legendrian knot $\Lambda$ satisfies $U\prec\Lambda$, then $U$ can be filled at the negative end by a Lagrangian disk. The result is a Lagrangian filling of $\Lambda$ with genus zero. Boileau and Orevkov showed that any such $\Lambda$ must be \emph{quasipositive} \cite{BoiOre01}. Furthermore, $\Lambda$ must be \emph{smoothly slice}, meaning it bounds a smooth disk in $B^4$. 

Cornwell, Ng, and Sivek \cite{CorNg16} show that no nontrivial $\Lambda$ can have a decomposable Lagrangian concordance to $U$. They also provide many obstructions to and examples of Lagrangian concordance. For example, if $\Lambda$ has at least two normal rulings, then $\Lambda\nprec U$. 

\begin{figure}
	\begin{center}
		\begin{tikzpicture}
			\node[inner sep=0] at (0,0) {\includegraphics[width=12 cm]{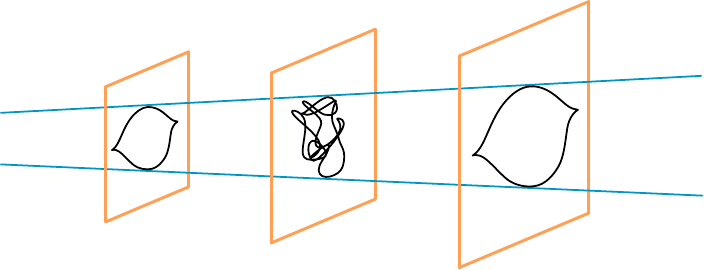}};
			\node at (-0.4,-1.1){$\Lambda$};
			\node at (3.2,-1.25){$U$};
			\node at (-3.5,-0.9){$U$};
			\node at (1.1,-0.1){$C_2$};
			\node at (-2,-0.1){$C_1$};
			\node at (-0.3,-1.75){$S^3$};
			\node at (3.2,-2){$S^3$};
			\node at (-3.4,-1.5){$S^3$};
			\node at (5,-2.5){$\RR_t\times S^3$};
		\end{tikzpicture}
		\caption{A Lagrangian concordance $C_1$ from $U$ to $\Lambda$ glued to a Lagrangian concordance $C_2$ from $\Lambda$ to $U$.}
		\label{fig:ulambdau}
	\end{center}
\end{figure}
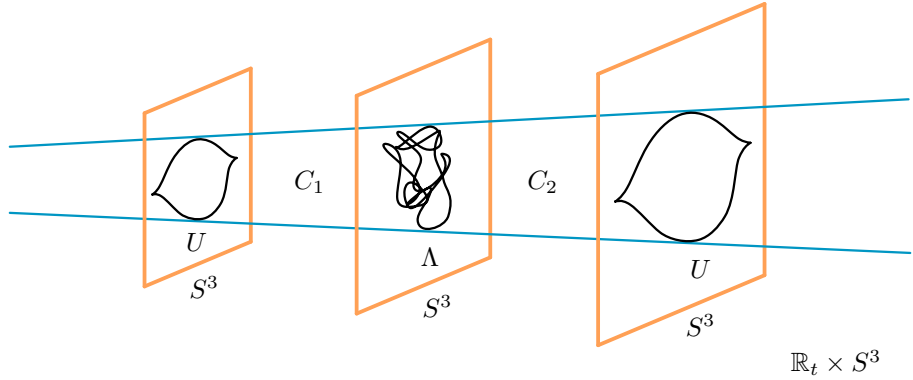

In this paper we develop new obstructions to the double Lagrangian concordance $U\prec\Lambda\prec U$ which do not depend on the Legendrian contact homology of $\Lambda$ itself. The main result is that we are able to obstruct the topological knot type of $\Lambda$:
\begin{thm}\label{thm:main}
	Let $U$ be the standard $tb=-1$ unknot. Let $\Lambda$ be a Legendrian knot satisfying $U\prec\Lambda\prec U$, and $\Lambda\neq U$. Then $\Lambda$ cannot be smoothly the closure of a 3-braid.
\end{thm}

The 3-braids are a relatively small class of braids: Murasugi classified representatives for all conjugacy classes of 3-braids \cite{Mur74}. The proof of Theorem \ref{thm:main} involves obstructing the closures of these braids. We prove some obstructions by reframing the problem of Lagrangian concordance as one of fillings of $\Sigma_p(\Lambda')$, the $p$-fold cyclic cover of $S^3$ branched over $\Lambda'$, a transverse push-off of $\Lambda$: 
\begin{thm}\label{thm:contact_hypersurface}
	If $\Lambda$ is a Legendrian knot which satisfies $U\prec\Lambda\prec U$, then any $p$-fold cyclic branched cover $\Sigma_p(\Lambda')$ of $S^3$ branched over the transverse push-off $\Lambda'$ of $\Lambda$ embeds as a contact type hypersurface in $(B^4,\xi_{std})$. Moreover, $\Sigma_p(\Lambda')$ is Stein fillable and has a Stein filling which embeds in $(B^4,\xi_{std})$.
\end{thm}

\begin{thm}\label{thm:fillings_embed}
	If $\Lambda$ is a Legendrian knot which satisfies $U\prec\Lambda\prec U$, then any filling of the $p$-fold cyclic branched cover $\Sigma_p(\Lambda')$ of $S^3$ branched over the transverse push-off $\Lambda'$ must embed in a blow up of $B^4$ and must have negative definite intersection form.
\end{thm}

We consider specifically the branched double cover to obtain a further obstructions using $d_3$, Gompf's 3-dimensional invariant on 2-plane fields \cite{Gom98}:
\begin{thm}\label{thm:d3_obstruction}
	If $\Lambda$ is a Legendrian knot which satisfies $U\prec\Lambda\prec U$, then the contact $2$-fold cyclic branched cover $(\Sigma_2(\Lambda'),\xi)$ of $S^3$ branched over the positive transverse push-off $\Lambda'$ of $\Lambda$ has 
	$$d_3(\xi) =-\frac{1}{2}.$$
\end{thm}

These obstructions eliminate all 3-braids except those which are quasipositive and of algebraic length 2. To obstruct the remaining 3-braids, we study particular fillings of their branched double covers by drawing the Weinstein handlebody diagrams \cite{Wei91} of these fillings. Weinstein handlebody diagrams are the symplectic analogue to Kirby diagrams and illustrate a symplectic version of handle decomposition. We obtain these diagrams by adapting the recipe laid out by Casals and Murphy in \cite{CasMur19} and prove the following theorem:

\begin{thm}\label{thm:diagram}
	Let $\Lambda\neq U$ be a Legendrian knot which is smoothly the closure of a quasipositive 3-braid of algebraic length 2. Let $\Lambda'$ be a positive transverse push off of $\Lambda$. If Lemma \ref{lemma:qgreaterthanp} gives $\gamma$ with slope $(p,q)$, $0< p < q$, then there is a filling of $\Sigma_2(\Lambda')$, the double cover of $S^3$ branched over $\Lambda'$, given by the handle decomposition consisting of a single 1-handle and a single 2-handle pictured in Figure \ref{fig:d_Lambda}.
\end{thm}

\begin{figure}
	\begin{center}
		\includegraphics[width=5 cm]{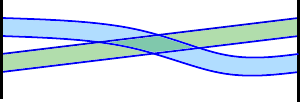}
	\end{center}
	\caption[A Weinstein filling of $\Sigma_2(\Lambda')$.]{A Weinstein filling of $\Sigma_2(\Lambda')$. The shaded blue and shaded green represent parallel strands of a single Legendrian link in $S^1\times S^2$.}
	\label{fig:d_Lambda}
\end{figure}

We then compute the Legendrian contact homology differential graded algebra of the Legendrian attaching sphere depicted in the general diagram. The homology of this differential graded algebra, called the Legendrian contact homology was introduced by Chekanov and Eliashberg \cite{Che02,Eli98}, and was the first non-classical Legendrian link invariant. The differential graded algebra is generated by Reeb chords in a Legendrian link in $\RR^3$, and the differential counts holomorphic disks in $\RR\times \RR^3$. A $\ZZ$-graded version of the Legendrian contact homology differential graded algebra was defined by Etnyre, Ng, and Sabloff \cite{EtnNg02}, and a version of this differential graded algebra was defined for links in $\#^m(S^1\times S^2)$ by Ekholm and Ng \cite{EkhNg15}, and simplified by Etg\"u and Lekili \cite{EtgLek19}. The Legendrian contact homology DGA can be used to compute invariants of the Weinstein manifolds including the symplectic homology of the filling of $\Sigma_2(\Lambda')$ following work of \cite{BouEkh12}. We prove:

\begin{thm}\label{thm:nonvanishingSH}
	Let $\Lambda\neq U$ be a Legendrian knot which is smoothly the closure of a quasipositive 3-braid of algebraic length 2. Let $\Lambda'$ be a positive transverse push off of $\Lambda$. Then there is a filling of $\Sigma_2(\Lambda')$, the double cover of $S^3$ branched over $\Lambda'$, which has nonvanishing symplectic homology.
\end{thm}

The main result (Theorem \ref{thm:main}) then follows from an argument showing that this filling of $\Sigma_2(\Lambda')$ cannot embed in $B^4$.

The strategy used in this paper can also be extended to $n$-stranded braids on a case by case basis. In example \ref{eg:m10_140}, we consider the $m(10_{140})$ knot, which is the closure of a 4-stranded braid. It was previously not known whether or not this knot can be Lagrangian concordant to the unknot since its Legendrian contact homology DGA is stable tame isomorphic to the DGA of the unknot \cite{CorNg16} and thus indistiguishable via contact homology techniques. We show that in fact that $m(10_{140})\nprec U$.

A corollary of Theorem \ref{thm:main} may be of separate interest. In \cite{MarTos20}, Mark and Tosun pose the following question: \emph{which smooth, oriented manifolds can be realized, up to diffeomorphism, as contact type hypersurfaces in $\RR^{2n}$ with the standard symplectic structure?} They rule out the Brieskorn spheres. We find another infinite family of contact manifolds which are rational homology spheres but do not embed as contact type hypersurfaces in $\RR^4$. 

\begin{cor}\label{cor:not_hyper}
	Let $\Sigma_2(\Lambda')$ be the double cover of $S^3$ branched over a quasipositive transverse knot which is the closure of a 3-braid of algebraic length 2. Suppose $\Lambda'$ is not the unknot. Then $\Sigma_2(\Lambda')$ does not embed as a contact type hypersurface in $\RR^4$.
\end{cor}

\subsection*{Structure of the paper}

In section 2, we build some obstructions from fillings of double covers and apply them to 3-braids. In section 3, we produce particular open book decompositions of the branched double covers of the knots in a remaining family of 3-braids. We find corresponding Weinstein Lefschetz fibrations of fillings of these covers. In section 4, we draw the Weinstein handlebody diagrams for these fillings. In Chapter 5, we compute the Legendrian contact homology differential graded algebra of the Legendrian attaching spheres of these Weinstein handlebody diagrams and prove that these fillings have non-zero symplectic homology. In section 6, we prove Theorem \ref{thm:main}.

\subsection*{Acknowledgements}

The author thanks Roger Casals, Orsola Capovilla-Searle, Jonathan Evans, Yank{\i} Lekili, Steven Sivek, Laura Starkston, Shea Vela-Vick, No\'emie Legout, and Russell Avdek for their helpful insights and conversations. This research was supported by the Engineering and Physical Sciences Research Council
[EP/L015234/1], the EPSRC Centre for Doctoral Training in Geometry and Number Theory (The London
School of Geometry and Number Theory), University College London, and by Louisiana State University. 

%% file: Chapter3.tex
\section{Initial obstructions from fillings}

In this section, we prove some new obstructions to Lagrangian concordance to and from the $tb=-1$ unknot which come from restrictions on the fillings of cyclic branched covers of transverse knots. We apply these to obstruct all but a particular family of closures of 3 braids from Lagrangian double concordance.

In \cite{Ben83}, Bennequin showed that any conjugacy class of braids can be closed in a natural way to produce a transverse knot in $(S^3,\xi_{std})$, and that every transverse knot is transversely isotopic to a closed braid. Orevkov, Shevchishin, and Wrinkle \cite{OreShe03,Wri02} show that two transverse closed braids that are isotopic as transverse knots are also isotopic as transverse braids. More precisely:
\begin{thm}\label{thm:tranverseMarkov} \cite{OreShe03,Wri02}
Two braids represent transversally isotopic links if and only if one can pass from one braid to another by conjugations in braid groups, positive Markov moves, and their inverses. 
\end{thm}
\emph{Markov moves}, also called \emph{positive or negative stabilizations}, are an operation on braids that increases the number of strands but preserves the topological knot type of the braids, see Figure \ref{fig:markov}. 

\begin{figure}
	\begin{center}
		\includegraphics[width=7 cm]{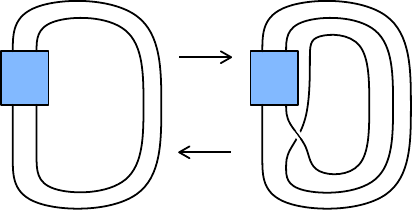}
	\end{center}
	\caption[A positive Markov move and its inverse.]{Going from left to right, a positive Markov move (or stabilization) and from right to left its inverse (or de-stabilization).}
	\label{fig:markov}
\end{figure}

Recall that a braid is \emph{quasipositive} if it is the product of conjugates of positive generators of the braid group. A knot is called \emph{quasipositive} if it is the closure of a quasipositive braid. It follows from \cite{BoiOre01} that Lagrangian fillable Legendrian knots are quasipositive. Using conjugations and positive Markov moves (and their inverses), Hayden proves the following:
\begin{thm}\label{thm:Hayden} \cite{Hay18}
Every quasipositive link has a quasipositive representative of minimal braid index.
\end{thm}

\begin{remark}\label{remark:quasipositive_transverse} 
Thus every transverse representative of a quasipositive link is transversely isotopic to a transverse quasipositive braid of minimal braid index. If a quasipositive knot has braid index 3, then any quasipositive transverse representative of that smooth knot type is transversely isotopic to a transverse 3-braid. For a candidate $\Lambda$ satisfying $U\prec\Lambda\prec U$, we obtain a transverse knot $\Lambda'$ which is a positive push off of $\Lambda$. If $\Lambda$ is smoothly the closure of a 3-braid, then so is $\Lambda'$. 
\end{remark}

\subsection{Embeddings of branched covers}

To begin, we cite Lemma 4.1 of \cite{CaoGal14} which gives a way of approximating a Lagrangian filling with a symplectic one:
\begin{lemma}\cite{CaoGal14}\label{lem:approximation}
Let $(X, \omega)$ be a strong filling of $(Y, \xi)$ with an oriented Lagrangian $L\subset X$ whose boundary is a Legendrian $\Lambda \subset Y$. Then, the Lagrangian surface $L$ may be $C^\infty$ approximated by a symplectic surface $L'$ that satisfies:
\begin{enumerate}
\item $\omega|_{L'} > 0$ and
\item $\partial L'$ is a positive transverse link smoothly isotopic to $\Lambda$. 
\end{enumerate}
\end{lemma}

This result follows from Lemma 2.3.A proved by Eliashberg \cite{Eli95}:
\begin{lemma}\cite{Eli95}
Let $F$ be a connected surface with boundary in a 4-dimensional symplectic manifold $(X,\omega)$. Suppose that $\omega|_F$ is positive near $\partial F$ and non-negative
elsewhere. Then $F$ can be $C^\infty$-approximated by a surface $F'$ which coincides with $F$ near $\partial F = \partial F'$ and such that $\omega|_{F'}>0$.
\end{lemma}

\begin{figure}
	\begin{center}
		\begin{tikzpicture}
			\node[inner sep=0] at (0,0) {\includegraphics[width=8 cm]{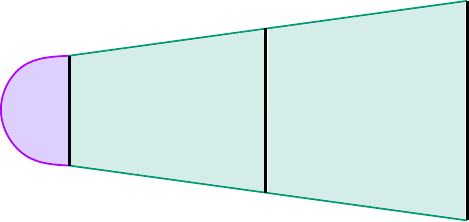}};
			\node at (-3.5,0){$B^4$};
			\node at (-2.8,-1.6){$\Sigma_p(U')$};
			\node at (0.5,-2){$\Sigma_p(\Lambda')$};
			\node at (3.8,-2.4){$\Sigma_p(U')$};
		\end{tikzpicture}
		\caption[$\Sigma_p(C')\cup B^4$, a filling of $\Sigma_p(U')$.]{$\Sigma_p(C')\cup B^4$, a filling of $\Sigma_p(U')$. $\Sigma_p(C')$ is represented by the green area.}
		\label{fig:filling}
	\end{center}
\end{figure}
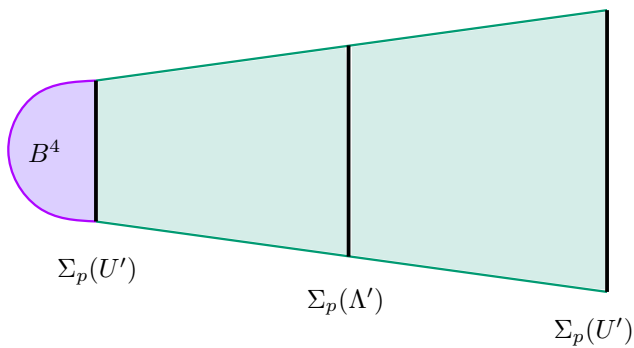

Additionally, we use Theorem 1.2 of \cite{Cha10}:
\begin{thm} \cite{Cha10} \label{thm:cylindrical}
Consider the standard contact $S^3$ and let $U$ be the standard Legendrian unknot with $tb(U)=-1$. Let $C$ be an oriented Lagrangian cobordism from $U$ to itself, then there is a compactly supported symplectomorphism $\phi$ of $\RR\times S^3$ such that $\phi(C) = \RR\times U$.
\end{thm}

This theorem follows from a result of Eliashberg and Polterovich, Theorem 1.1.A of \cite{EliPol96}:

\begin{thm} \cite{EliPol96}
Any flat at infinity Lagrangian embedding of $\RR^2$ into the standard symplectic $\RR^4$ is isotopic to the flat embedding via an ambient compactly supported Hamiltonian isotopy of $\RR^4$.
\end{thm}

With these we can prove the following two obstructions:

\newtheorem*{thm:contact_hypersurface}{Theorem \ref{thm:contact_hypersurface}}
\begin{thm:contact_hypersurface}
	If $\Lambda$ is a Legendrian knot which satisfies $U\prec\Lambda\prec U$, then any $p$-fold cyclic branched cover $\Sigma_p(\Lambda')$ of $S^3$ branched over the transverse push-off $\Lambda'$ of $\Lambda$ embeds as a contact type hypersurface in $(B^4,\xi_{std})$. Moreover, $\Sigma_p(\Lambda')$ is Stein fillable and has a Stein filling which embeds in $(B^4,\xi_{std})$.
\end{thm:contact_hypersurface}

Note that since $\Lambda$ must be quasipositive, fillability of $\Sigma_p(\Lambda')$ also follows from a result of \cite{Pla06}.

\begin{proof}[Proof of Theorem \ref{thm:contact_hypersurface}]
Suppose we have some Legendrian knot $\Lambda$ in $S^3$ such that $U\prec\Lambda\prec U$ in $\RR_t\times S^3$ (which is symplectomorphic to $\RR^4\setminus B^4$). Suppose $\Lambda \in \{0\}\times S^3$. Let $C_1$ be the Lagrangian cylinder from $U$ to $\Lambda$ and $C_2$ be the Lagrangian cylinder from $\Lambda$ to $U$, see Figure \ref{fig:ulambdau}. By Theorem \ref{thm:cylindrical}, $C_1\cup C_2$ is Hamiltonian isotopic to the product cylinder $\RR\times U$.

Fill in $C_1$ at the negative end by a Lagrangian disk in $B^4$. Then $C\defeq D^2\cup C_1\cup C_2$ is a Lagrangian filling of the unknot $U$ by a standardly embedded disk. By Lemma \ref{lem:approximation}, there is a symplectic approximation of $C$, call it $\phi(C) = C'$ with transverse boundary $\partial C' = U'$ where $U'$ is the transverse unknot with self linking number $-1$.  The $p$-fold cyclic branched cover of $S^3$ branched over $U'$ is the standard $S^3$ (see Lemma 2.4 of \cite{CasEtn19}).

The $p$-fold cyclic branched cover of $\RR_t\times S^3$ branched over $C'$, which we will call $\Sigma_p(C')$ is the standard four ball, $(B^4,\xi_{std})$. For $\phi$ a small enough perturbation, there is some $t$ in a neighbourhood of 0 such that $C'$ is transverse to $\{t\}\times S^3$, and $\Lambda'=C'\cap (\{t\}\times S^3)$ is a transverse push-off of $\Lambda$ in $\{t\}\times S^3$. Then $\Sigma_p(\Lambda')$, the $p$-fold cyclic branched cover of $S^3$, branched over $\Lambda'$ is a contact type hypersurface in $(B^4,\xi_{std})$. We illustrate this in Figure \ref{fig:filling}.

Taking the negative end of $\Sigma_p(C')$ bounded by $\Sigma_p(\Lambda')$ gives a Stein filling $X$ of $\Sigma_p(\Lambda')$ which embeds in $(B^4,\xi_{std})$. 
\end{proof}

\begin{figure}
	\begin{center}
		\begin{tikzpicture}
			\node[inner sep=0] at (0,0) {\includegraphics[width=6 cm]{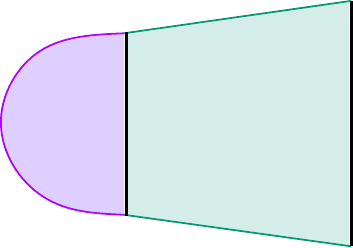}};
			\node at (-2,0){$X$};
			\node at (1,0){$V$};
			\node at (-0.5,-2.1){$\Sigma_p(\Lambda')$};
			\node at (3,-2.6){$\Sigma_p(U')$};
		\end{tikzpicture}
		\caption[$V \cup X$, a filling of $\Sigma_p(U')$]{$V \cup X$, a filling of $\Sigma_p(U')$, where $X$ is a filling of $\Sigma_p(\Lambda')$.}
		\label{fig:halffilling}
	\end{center}
\end{figure}
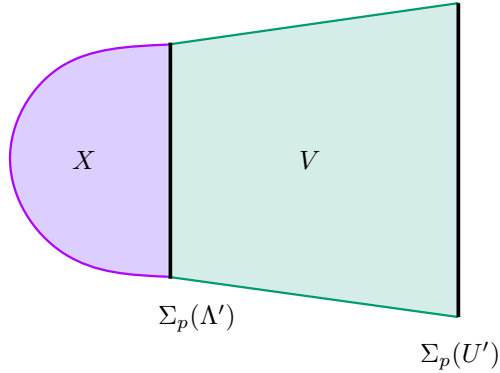

\newtheorem*{thm:fillings_embed}{Theorem \ref{thm:fillings_embed}}
\begin{thm:fillings_embed}
	If $\Lambda$ is a Legendrian knot which satisfies $U\prec\Lambda\prec U$, then any filling of the $p$-fold cyclic branched cover $\Sigma_p(\Lambda')$ of $S^3$ branched over the transverse push-off $\Lambda'$ must embed in a blow up of $B^4$ and must have negative definite intersection form.
\end{thm:fillings_embed}

\begin{proof}
Suppose we have a Legendrian knot $\Lambda$ in $S^3$ such that $\Lambda\prec U$ in $\RR_t\times S^3$. Recall the construction from Theorem \ref{thm:contact_hypersurface}, of the branched cover $\Sigma_p(C')$. Recall also the 4-manifold $X$ which is a filling of the 3-manifold $\Sigma_p(\Lambda')$, and which embeds in $\Sigma_p(C')$. Let 
$$V\defeq  \Sigma_p(C') \setminus X.$$ 

Let $X_0$ be any filling of $\Sigma_p(\Lambda')$ and let $W = X_0\cup V$ be obtained by gluing $X_0$ to $V$ along $\Sigma_p(\Lambda')$. The boundary $\partial(W) = \Sigma_p(U')=S^3$. By a theorem of McDuff (Theorem 1.7 of \cite{McD90}) or of Gromov (p 311 of \cite{Gro85}), $W$ is necessarily a blowup of $B^4$ with the standard symplectic structure, as $S^3$ has a unique minimal filling, $B^4$.

Finally, since any surface embedded in $X$ is embedded in $B^4\#n\overline{\cptwo}$ for some $n\geq 0$, $X$ must be negative definite.
\end{proof}

\subsection{Using Gompf's 3-dimensional 2-plane field invariant}

Next, we prove an obstruction to the concordance $U \prec\Lambda\prec U$ coming from the $d_3$ invariant of $\Sigma_2(\Lambda')$, proved in Theorem 4.16 of \cite{Gom98}, and apply it to the case where $\Lambda'$ is a transverse knot which is the closure of a 3-braid.

For a given contact 3-manifold $(M,\xi)$, we define the $d_3$ invariant on the homotopy type of the 2-plane field:
\begin{thm}\cite{Gom98}
Suppose we have a contact 3-manifold $(M,\xi)$, and suppose we have an almost complex 4-manifold $(X,J)$ such that $\partial X = M$, with $\xi$ induced by the complex tangencies: $\xi = (TM)\cap J(TM)$. Let $\sigma(X)$ denote the signature of $X$, and let $\chi(X)$ denote the Euler characteristic of $X$. For $c_1(\xi)$ a torsion class, the rational number 
$$d_3(\xi) = \frac{c_1^2(X,J)-3\sigma(X)-2\chi(X)}{4}$$ 
is an invariant of the homotopy type of the 2-plane field $\xi$. 
\end{thm}

Using Theorem \ref{thm:contact_hypersurface}, we prove a restriction on $d_3(\xi)$ for the branched double cover $(\Sigma_2(\Lambda'),\xi)$: 
\newtheorem*{thm:d3_obstruction}{Theorem \ref{thm:d3_obstruction}}
\begin{thm:d3_obstruction}
	If $\Lambda$ is a Legendrian knot which satisfies $U\prec\Lambda\prec U$, then the contact cyclic branched double cover $(\Sigma_2(\Lambda'),\xi)$ of $S^3$ branched over the transverse push-off $\Lambda'$ of $\Lambda$ has 
	$$d_3(\xi) =-\frac{1}{2}.$$
\end{thm:d3_obstruction}

\begin{proof}
Suppose we have a Legendrian knot $\Lambda$ satisfying $U\prec\Lambda\prec U$. Consider a transverse push-off $\Lambda'$ in $S^3$ of $\Lambda$ and $\Sigma_2(\Lambda')$, the  cyclic double cover of $S^3$ branched over $\Lambda'$. By Theorem \ref{thm:contact_hypersurface}, there is a filling $(X, \omega)$ of $\Sigma_2(\Lambda')$ which embeds in $\RR^4$.

Now we will compute $d_3(\xi)$ using $X$. We begin with $\sigma(X)$, which is the signature of the intersection form $Q_X$. Let $S$ be any surface in $X$ and take the class $[S]\in H_2(X)$. Then since $S$ embeds in $X$ and $X$ embeds in $\RR^4$, $S$ embeds in $\RR^4$ and has self-intersection 0. Thus $Q_X\equiv 0$, and $\sigma(X)=0$.

Next we want to find $c_1^2(X,J)$. We know $c_1(X,J)\in H^2(X,\ZZ)$ and by Poincar\'e duality, $H^2(X)\cong H_2(X,\partial X)$. Thus, let us consider a properly embedded surface $F\subset X$ such that 
$$[F] = PD(c_1(X,J))\in H_2(X,\partial X).$$
Consider the exact sequence
$$0\rightarrow H_2(X) \rightarrow H_2(X,\partial X) \rightarrow H_1(\partial X).$$

Note that for $\Sigma_2 (K)$, the branched double cover of a knot $K$, $|H_1(\Sigma_2 (K))|=\det(K)$ where the determinant of a knot is the Alexander polynomial evaluated at $-1$.
Thus $H_1(\partial X)$ is finite and $[\partial F]\in H_1(\partial X)$, thus 
$$\partial(d[F])=0$$
for some $d>0$, for instance $d=\det(K)$, in $H_1(\partial X)$. So $d[F]=[S]$ for some closed surface $S\subset X$, and 
$$[F]^2=[F]\cdot\frac{1}{d}[S] =\frac{Q_X(F,S)}{d}\in \frac{1}{d}\ZZ.$$

Since $Q_X=0$, we have $c_1(X,J)^2=[F]^2=0$.

Finally, we compute the Euler characteristic of $X$. Note that $X$ is the branched double cover of a properly embedded disk $D$ in $B^4$, $X=\Sigma_2(D)$. Thus if we let $\nu(D)$ denote a neighbourhood of $D$ in $X$, we compute:
$$\chi(X)=2\chi(B^4\setminus \nu(D))+\chi(\nu(D))-\chi(B)$$
where $B$ is a circle bundle over $D$. Thus,
$$\chi(X)=2\cdot0+1-0=1.$$
Thus,
\begin{align*}
	d_3(\xi) &=\frac{c_1^2(X,J)-3\sigma(X)-2\chi(X)}{4}
	=\frac{0-3(0)-2(1)}{4}
	=-\frac{1}{2}.\qedhere
\end{align*}
\end{proof}

\begin{remark}
	Note that Theorem \ref{thm:contact_hypersurface} and Theorem \ref{thm:fillings_embed} also hold in the case that $\Lambda$ is a Legendrian link if the Lagrangian obtained by gluing the Lagrangian cobordisms from $U$ to $\Lambda$ and from $\Lambda$ to $U$ is a cylinder. Thus so does Theorem \ref{thm:d3_obstruction}.
\end{remark}

Theorem 1.2 of \cite{Pla06} states that for $\Lambda'$ a transverse knot and $(\Sigma_2(\Lambda'),\xi)$ the double cover of $S^3$ branched over $\Lambda'$, $d_3(\xi)$ is completely determined by the topological link type of $\Lambda'$ and its self-linking number sl$(\Lambda')$. To this end, Ito \cite{Ito17} provides a formula:

\begin{thm}\cite{Ito17}\label{thm:Ito_d3}
If a contact 3-manifold $(M,\xi)$ is a $p$-fold cyclic contact branched covering of $(S_3,\xi_{std})$ branched along a transverse link $K$, then
$$d_3(\xi)= -\frac{3}{4}\sum_{\omega:\;\omega^p=1}\sigma_\omega(K)-\frac{p-1}{2} sl(K) -\frac{1}{2}p$$
where $\sigma_\omega(K)$ is the Tristam-Levine signature of $K$, given by $(1-\omega)A+(1-\overline{\omega})A^T$ where $A$ is the Seifert matrix of $K$, and $sl(K)$ is the self-linking number of $K$.
\end{thm}

We apply this formula to 3-braids:

\begin{thm} \label{cor:alg_length}
	If $\Lambda$ is a Legendrian knot which satisfies $U\prec\Lambda\prec U$, and $\Lambda'$, the transverse pushoff of $\Lambda$, is the closure of a 3-braid $\beta$, then the algebraic length of $\beta$ is 2.
\end{thm}

\begin{proof}
Let $\Lambda$ be a Legendrian knot which satisfies $U\prec\Lambda\prec U$, and $\Lambda'$ be a transverse pushoff of $\Lambda$. Suppose $\Lambda'$ is the closure of a 3-braid. Let $(\Sigma_2(\Lambda'), \xi)$ be a cyclic double cover of $S^3$ branched over $\Lambda'$. Then from \ref{thm:Ito_d3} and \ref{thm:d3_obstruction}, we get that 
$$-\frac{1}{2}=d_3(\xi)= -\frac{3}{4}\sum_{\omega:\omega^2=1}\sigma_\omega(\Lambda')-\frac{2-1}{2} sl(\Lambda') -\frac{1}{2}2.$$	
Since $\Lambda$ is slice, its signature must vanish.
$$-\frac{1}{2}= 0-\frac{1}{2} sl(\Lambda') -1.$$	
The self-linking number of $\Lambda'$ can be computed from this algebraic length and is given by:
$$sl(\Lambda') = \text{len}(\beta)-(\text{number of strands of }\beta)=\text{len}(\beta)-3.$$
Thus $\text{len}(\beta)= -1 +3 = 2$.
\end{proof}

Using the same strategy, we can obtain a more general result for $n$-stranded braids:

\begin{cor} \label{cor:n_braid}
	If $\Lambda$ is a Legendrian knot which satisfies $U\prec\Lambda\prec U$, and $\Lambda'$ is the closure of n $n$-braid $\beta$, then $\Lambda'$ has self linking number $-1$ and the algebraic length of $\beta$ is $n-1$.
\end{cor}

\begin{proof}
	Consider the cyclic double cover of $S^3$ branched over $\Lambda'$, $(\Sigma_2(\Lambda'),\xi)$. We apply the result of Theorem \ref{thm:d3_obstruction} and the fact that the signature must vanish to Ito's formula.
	\begin{align*}
		d_3(\xi)&= -\frac{3}{4}\sum_{\omega:\;\omega^2=1}\sigma_\omega(K)-\frac{2-1}{2} sl(K) -\frac{1}{2}2\\
		-\frac{1}{2}&= 0-\frac{1}{2} sl(\Lambda') -1\\
		-1&= sl(\Lambda')= len(\beta)-n
		 \qedhere
	\end{align*}
\end{proof}

%% file: Chapter4.tex
\section{Open book decompositions and Weinstein Lefschetz fibrations}
\label{OBDtoWD}

In this section, we find a Weinstein domain $X_\Lambda$ whose boundary is $\Sigma_2(\Lambda')$ where $\Lambda'$ is the closure of a 3-braid with algebraic length 2. For background on Weinstein manifolds and their handlebody diagrams, see \cite{Wei91, CieEli12} or the summary in \cite[Section 2]{AcuCap21}.

To find this Weinstein domain, we begin by expressing $\Sigma_2(\Lambda')$ as a particular \emph{open book decomposition}, a  pair $(F,\phi)$ where $F$ is a surface with non-empty boundary and the \emph{monodromy} $\phi$ is a diffeomorphism of $F$ with $\phi|_{\partial F} = \text{id}$. $\phi$ is given by compositions of Dehn twists about some curves in $F$. Let $\alpha_1$ and $\alpha_2$ be some such curves. We denote the composition of Dehn twists about these curves as:
$$\tau_{\alpha_1\alpha_2} = \tau_{\alpha_1}\circ\tau_{\alpha_2}=\tau_{\alpha_1}\tau_{\alpha_2}$$ 
If $\alpha$ is a curve and $f$ is an orientation preserving surface homeomorphism then $\tau_{f(\alpha)} =f\circ \tau_\alpha\circ f^{-1}$.

From this particular open book decomposition, we construct a \emph{Weinstein Lefschetz fibration}:

\begin{dfn} 
	Let $X$ be a compact, oriented symplectic 4-manifold,  let $\Sigma$ be a compact oriented manifold with dimension $2$. \emph{A Lefschetz fibration} $\pi:X\to \Sigma$ is a smooth surjective map which is a locally trivial fibration except at finitely many isolated, nondegenerate critical points with distinct values on the interior of $\Sigma$. In local coordinates near a critical point, the fibration is modelled by $\pi (z_1, z_2)=z_1^2+z_2^2$. The Lefschetz fibration is \emph{symplectic} if the fibers are symplectic submanifolds. 
	
	Let $x_0 \in \Sigma$ be a critical value. Let $x\in \Sigma$ be a generic value. Take a path $x\to x_0\subset \Sigma$. In a fiber $\pi^{-1}(x)$ of a generic value $x$, there is a closed curve $C$ called a \emph{vanishing cycle} which collapses after parallel transport along this path. $\pi^{-1}(x_0)$ can be identified with $\pi^{-1}(x)$ by collapsing $C$ to singular ordinary double point.
	
	A Lefschetz fibration $\pi:X\to \Sigma$ where $X$ is a Weinstein domain is a \emph{Weinstein Lefschetz fibration} if its generic fiber $F$ is also a Weinstein domain and if $X$ is obtained by attaching critical Weinstein handles along attaching Legendrians $\Lambda_i\subset F\times S^1\subset \partial(F\times D^2)$ obtained by lifting the vanishing cycles $C_i\subset F$. 
\end{dfn}

\begin{remark}\label{remark:Seidel}
	Suppose we have an open book decomposition $\pi:M\to S^1$ of a contact 3-manifold $(M,\xi)$ with monodromy given by positive Dehn twists $\tau_{\alpha_1},\dots,\tau_{\alpha_k}$ about some curves $\alpha_1,\dots,\alpha_k$ in $F$. Then we can build a Lefschetz fibration $\pi:X\to D^2$ of the Stein filling $X$ of $M$ where total monodromy, given as the composition of the $\tau_i$ Dehn twists about vanishing cycles $C_i$ collapsing in the critical surfaces $\pi^{-1}(x_i)$, must agree with the monodromy of the open book. As long as the fiber $F$ has a Weinstein structure, $X$ also has a Weinstein structure corresponding to the attachment of critical Weinstein handles along the corresponding vanishing cycles. 
\end{remark} 

To begin our construction, we prove the following lemma:

\begin{lemma}\label{lem:sigma1braid}
Any quasipositive 3-braid of algebraic length 2 is conjugate to $\sigma_1 B \sigma_1 B^{-1}$ for some braid $B$. 
\end{lemma}

\begin{proof}
Suppose we have some quasipositive 3-braid of algebraic length 2, call it $A$. Then $A$ is the product of two conjugates of positive generators of the braid group, 
$$A = B_i \sigma_iB_i^{-1}B_j\sigma_jB_j^{-1}$$
for $i,j\in\{1,2\}$. For $k\in \{i,j\}$, if $k=1$, let $B_k = B_k'$, and if $k=2$, since $\sigma_2 = \sigma_1^{-1}\sigma_2^{-1}\sigma_1\sigma_2\sigma_1$, let $B_k' = B_k\sigma_1^{-1}\sigma_2^{-1}$.  Then
$$A = B_1'\sigma_1B_1'^{-1}B_2'\sigma_1B_2'^{-1}$$
Now we can conjugate by $B_1'^{-1}$:
\begin{align*}
	A	&= B_1'^{-1}(B_1'\sigma_1B_1'^{-1}B_2'\sigma_1B_2'^{-1})B_1'\\
			&=\sigma_1B_1'^{-1}B_2'\sigma_1B_2'^{-1}B_1'
\end{align*} 
So we take $B=B_1'^{-1}B_2'$.
\end{proof}

\begin{prop}\label{prop:OBD}
Let $\Lambda'$ be a transverse knot which is smoothly the closure of a quasipositive 3-braid of algebraic length 2. The double cover of $S^3$ branched over $\Lambda'$, $(\Sigma_2(\Lambda'),\xi)$ has an open book decomposition $(F,\phi)$ where $F$ is a torus with one boundary component, and $\phi=\tau_\alpha\tau_\gamma$. Here, $\alpha$ is the curve of slope $(1,0)$ and $\gamma$ is some essential simple closed curve in $F$, see Figure \ref{fig:curvesontorus}.
\end{prop}

\begin{figure}
	\begin{center}
		\begin{tikzpicture}
			\node[inner sep=0] at (0,0) {\includegraphics[width=12cm]{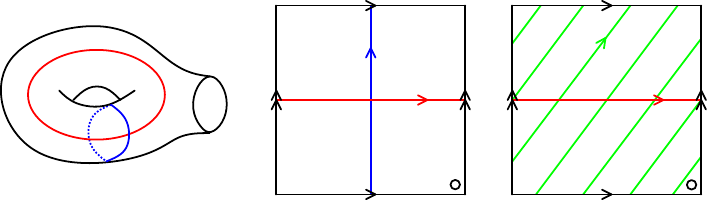}};
			\node at (-0.8,0.3){$\alpha$};
			\node at (0,1){$\beta$};
		\end{tikzpicture}
		\caption[Curves on the torus.]{The curves $\alpha$ (red), $\beta$ (blue), and an example of a curve $\gamma$ (green) on the torus with one boundary component. Here $\gamma$ is the curve of slope $(3,4)$.}
		\label{fig:curvesontorus}
	\end{center}
\end{figure}

\begin{proof}
To prove this, we will use the construction from Plamenevskaya \cite{Pla06}. When a contact manifold $(M,\xi)$ is the branched double cover of a transverse link $L$ which is the closure of a $(2k+1)$-stranded braid given by a braid word on $2k$ generators and their inverses, $\sigma_1, \sigma_1^{-1}, \dots, \sigma_{2k}, \sigma_{2k}^{-1}$, we can think of $L$ as a transverse link in the standard contact structure $\xi_{std}$ on $S^3$. $S^3$ has a planar open book decomposition with trivial monodromy and we arrange it so that the pages are transverse to the braid $L$. Then we may lift the contact structure on the double cover $\Sigma_2(L)$ of $S^3$ branched over $L$. The contact structure is compatible with the open book decomposition $(F_{k,1},\phi)$ where $F_{k,1}$ is the genus $k$, 1 boundary component surface obtained by taking each planar page branched over the three points where they intersect $L$ (see Figure \ref{fig:from_planar}), and $\phi$ is given by Dehn twists corresponding to the braid word which are the lifts of the half twists of $L$ in $S^3$. 

Thus the monodromy of the open book comes from the braid monodromy: each generator $\sigma_i$ in the braid corresponds to a Dehn twist $\tau_i$ along a curve in $F_{k,1}$. 

\begin{figure}
	\begin{center}
		\includegraphics[width=6cm]{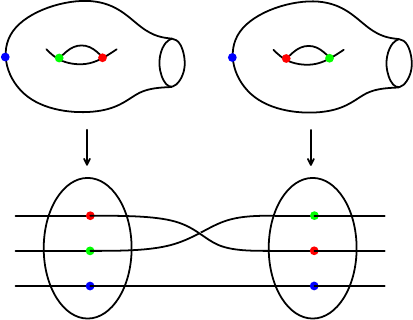}
		\caption[The double cover of a planar page branched over 3 points]{A torus with one boundary component is the double cover of a planar page branched over 3 points. The effect of a half twist in the braid lifts to a Dehn twist in the torus.}
		\label{fig:from_planar}
	\end{center}
\end{figure}

Let $\Lambda'$ be a transverse knot which is the closure of a quasipositive 3-braid of algebraic length 2. The double cover of $S^3$ branched over $\Lambda'$, $(\Sigma_2(\Lambda'),\xi)$ has an open book decomposition $(F,\phi)$.

Since $\Lambda'$ is a 3-braid, $F \defeq F_{1,1}$ is a torus with one boundary component with $\sigma_1$ corresponding to a Dehn twist about $\alpha$ and $\sigma_2$ corresponding to a Dehn twist about $\beta$ as seen in Figure \ref{fig:curvesontorus}.
$\phi$ is given by Dehn twisting along these curves corresponding to the braid monodromy of $\Lambda'$. By Lemma \ref{lem:sigma1braid}, $\Lambda'$ is the closure of of a braid of the form $\sigma_1 B \sigma_1 B^{-1}$ for some braid $B$ up to conjugation. The first $\sigma_1$ corresponds to a Dehn twist about $\alpha$, $\tau_\alpha$. The braid $B \sigma_1 B^{-1}$ corresponds to a series of Dehn twists 
$$\mu_1\circ\dots\circ\mu_m\circ\tau_\alpha\circ\mu_m^{-1}\circ\dots\circ\mu_1^{-1}$$
where the $\mu_i\in\{\tau_{\alpha}^{\pm1},\tau_{\beta}^{\pm1}\}$  correspond to the braid generators in $B$. Then since 
\begin{align*}
 &(\mu_1\circ\dots\circ\mu_m)\circ\tau_\alpha\circ(\mu_m^{-1}\circ\dots\circ\mu_1^{-1})\\
 =&(\mu_1\dots\mu_m)\circ\tau_\alpha\circ(\mu_1\dots\mu_m)^{-1}\\
 =&\tau_{\mu_1\dots\mu_m(\alpha)},
\end{align*}
we choose $\gamma \defeq \mu_1\dots\mu_m(\alpha)$. 

Thus the monodromy $\phi$ of the open book decomposition is given by Dehn twists along $\alpha$, the curve of slope $(1,0)$, and some essential simple closed curve $\gamma$.
\end{proof}

\begin{cor}\label{cor:WeinsteinLF}
Let $\Lambda'$ be a transverse knot which is smoothly the closure of a quasipositive 3-braid of algebraic length 2. The double cover of $S^3$ branched over $\Lambda'$, $(\Sigma_2(\Lambda'),\xi)$ has Weinstein filling $X_\Lambda$ with the following property. $X_\Lambda$ has a Weinstein Lefschetz fibration $\pi:X_\Lambda\rightarrow D^2$ with generic fiber $F$, a torus with one boundary component and monodromy given by vanishing cycles which are Dehn twists along the curves $\alpha$ and $\gamma$ where $\alpha$ is the $(1,0)$ curve and $\gamma$ is some essential simple closed curve in $F$, see Figure \ref{fig:curvesontorus}.
\end{cor}

\begin{proof}
Let $\Lambda'$ be a transverse knot which is smoothly the closure of a quasipositive 3-braid of algebraic length 2. By Proposition \ref{prop:OBD}, the double cover of $S^3$ branched over $\Lambda'$, $(\Sigma_2(\Lambda'),\xi)$ has an open book decomposition $(F,\phi)$ where $F$ is a torus with one boundary component, and $\phi$ is given by Dehn twisting along the curves $\alpha$ and $\gamma$ where $\alpha$ is the $(1,0)$ curve and $\gamma$ is some essential simple closed curve in $F$, see Figure \ref{fig:curvesontorus}. Then we may construct a Weinstein Lefschetz fibration by attaching handles along the lifts of vanishing cycles corresponding to $\alpha$ and $\gamma$ in a generic fiber $F$, as described in Remark \ref{remark:Seidel}.
\end{proof}

\begin{remark}\label{remark:TorusDtwistcombinatorics}
Repeated Dehn twists $\tau_\alpha$, $\tau_\beta$, $\tau_{\alpha}^{-1}$ and $\tau_{\beta}^{-1}$ about the curve $\alpha$ of slope $(1,0)$ and $\beta$ of slope $(0,1)$ in Figure \ref{fig:curvesontorus} transforms the curve $\alpha$ to a curve with some slope $(p,q)$ in the torus with one boundary component $F_{1,1}$. Indeed, the mapping class group of the torus with one boundary component is generated by the positive Dehn twists $\tau_\alpha$ and $\tau_\beta$ with presentation
$$\langle \tau_\alpha,\tau_\beta \mid \tau_\alpha\tau_\beta\tau_\alpha=\tau_\beta\tau_\alpha\tau_\beta\rangle,$$
see \cite{FarMar11} for details. Any diffeomorphism of $F_{1,1}$ generated by $\tau_\alpha$ and $\tau_\beta$ is equivalent to a Dehn twist along some non-separating curve with rational slope 
$$\frac{p}{q}\in\QQ\mathbb{P}^1 = \QQ\cup\{\infty\}.$$ 

\begin{table}[ht]
	\centering
	\caption{Effect of a Dehn twist about $\alpha$ or $\beta$ on $(p,q)$ in the torus with one boundary component. \label{table:pq}}
	\begin{tabular}[t]{lcc}
		\toprule
		Braid Generator & Dehn twist & Effect on $(p,q)$\\
		\midrule
		$\sigma_1$		&$\tau_\alpha$		& $(p,q)\mapsto (p+|q|,q)$\\
		$\sigma_1^{-1}$	&$\tau_\alpha^{-1}$	& $(p,q)\mapsto (p-|q|,q)$\\
		$\sigma_2$		&$\tau_\beta$		& $(p,q)\mapsto (p,q-|p|)$\\
		$\sigma_2^{-1}$	&$\tau_\beta^{-1}$	& $(p,q)\mapsto (p,q+|p|)$\\
		\bottomrule
	\end{tabular}
\end{table}

Applying an additional Dehn twist about $\alpha$ or $\beta$ has the result on $(p,q)$ as explained in Table \ref{table:pq}. Thus, given any $(p,q)$ with $p$ and $q$ relatively prime, a $(p,q)$ curve can be obtained by applying $\tau_{\alpha}^{\pm 1}$ and $\tau_{\beta}^{\pm 1}$ to $(1,0)$ according to the Euclidean algorithm. For instance if $q>p >0$, the Euclidean algorithm gives:
\begin{align*}
	q&=n_kp+r_k\\
	p&=n_{k-1}r_k+r_{k-1}\\
	r_k&=n_{k-2}r_{k-1}+r_{k-2}\\
	&\dots\\
	r_3&=n_3r_2+r_1\\
	r_2&=n_2r_1+1\\
	r_1&=n_1.
\end{align*}
So if $\gamma$ is the curve of slope $(p,q)$ in the torus with one boundary component, 
$$\gamma =\tau_{\beta}^{-n_k}\circ\tau_\alpha^{n_{k-1}}\circ\tau_{\beta}^{-n_{k-2}}\circ\dots\circ\tau_{\beta}^{-n_1}(\alpha).$$
\end{remark}

\begin{lemma}\label{lemma:qgreaterthanp}
Let $\Lambda'$ be a transverse knot which is not an unknot. Suppose $\Lambda'$ is smoothly the closure of a quasipositive 3-braid of algebraic length 2, $B$. Then we can always find a braid $B'$ equivalent to $B$ up to conjugation such that the procedure in Proposition \ref{prop:OBD} yields an open book decomposition with monodromy given by Dehn twists along the curves $\alpha$ and $\gamma$ of slopes $(1,0)$ and $(p,q)$ respectively, where $0<p<q$.
\end{lemma}

\begin{proof}
By Lemma \ref{lem:sigma1braid}, we know $\Lambda$ is the closure of a braid of the form $\sigma_1 B_0 \sigma_1 B_0^{-1}$, and that after applying Proposition \ref{prop:OBD}, $(\Sigma_2(\Lambda'),\xi)$ corresponds to an open book decomposition $(F,\phi)$ where $\phi$ is given by a Dehn twist on $\alpha$ and on $\gamma$, where 
$$\gamma=(\mu_1\dots\mu_m)(\alpha)=\mu_1\circ\dots\circ\mu_m(\alpha),$$
where the $\mu_i\in\{\tau_{\alpha}^{\pm1},\tau_{\beta}^{\pm1}\}$ correspond to the braid generators in $B_0$. 

By Remark \ref{remark:TorusDtwistcombinatorics}, $\gamma$ corresponds to a curve of slope $\frac{p}{q}\in \QQ \mathbb{P}^1$. Without loss of generality, we may assume that $q\geq0$. 

We know $q\neq 0$ since $(p,q)=(1,0)$ corresponds to 
$$\Sigma_2(\Lambda')\cong(S^1\times S^2)\#\RR\mathbb{P}^3,$$ 
but $\Sigma_2(\Lambda')$ must be a rational homology sphere so there is no such $\Lambda'$. 

If $p=1$, then if either $q=0$ or $q=1$, then 
$$\Sigma_2(\Lambda')\cong S^3,$$ and so $\Lambda'$ is the unknot. Thus $q\geq2$. Notice the following equivalence of braids:
\begin{align*}
	\sigma_1 B_0 \sigma_1 B_0^{-1}=&\sigma_1^k(\sigma_1 B_0\sigma_1 B_0^{-1})\sigma_1^{-k}\\
	=&\sigma_1(\sigma_1^k B_0)\sigma_1(B_0^{-1}\sigma_1^{-k}).
\end{align*}
for any $k\in \ZZ$. The braid $\sigma_1^kB_0$ corresponds to the Dehn twists 
$$\tau_\alpha^k\circ\mu_1\circ\dots\circ\mu_m$$ 
and the braid $(\sigma_1^k B_0)\sigma_1 (B_0^{-1}\sigma_1^{-k})$ corresponds to the Dehn twists
$$(\tau_\alpha^k\circ\mu_1\circ\dots\circ\mu_m)\circ\tau_\alpha\circ(\mu_m^{-1}\circ\dots\circ\mu_1^{-1}\circ \tau_\alpha^{-k})= \tau_\alpha^k\mu_1\dots\mu_m(\alpha).$$ 
Since $\gamma=\mu_1\dots\mu_m(\alpha)$ corresponds to a $(p,q)$ curve, applying an additional $k$ instances of $\tau_\alpha$ to $\mu_1\dots\mu_m(\alpha)$ corresponds to the curve of slope $(p+kq,q).$

Thus, we can replace $B_0$ with the braid $B=\sigma_1^kB_0$ where $k\in \ZZ$ satisfies $0<p+kq<q$. We have an equivalence of braids up to conjugation:
$$\sigma_1 B_0 \sigma_1 B_0^{-1} = \sigma_1 B\sigma_1 B^{-1},$$
and $B\sigma_1B^{-1}$ corresponds to a curve $\gamma$ with slope $\frac{p'}{q}$ with 
$0 < p'=p+kq < q.$ \qedhere
\end{proof}

\section{Drawing Weinstein diagrams}\label{sec:drawing_diagrams}

We now want to obtain a Weinstein handlebody diagram in Gompf standard form \cite{Gom98} for $X_\Lambda$ of the previous corollary. To obtain the attaching spheres the 2-handles we will use the recipe laid out by Casals and Murphy \cite{CasMur19}, which uses the following proposition to determine the attaching maps as simplified contact surgery curves on the boundary:

\begin{prop}\cite{CasMur19}\label{prop:CM}
Let $(Y,\xi)$ be a contact manifold with open book decomposition $(F,\phi)$ and let $\lambda$ denote the Liouville form on $F$. Let $S,L\subset(F,\lambda)$ be two exact Lagrangian submanifolds such that $S$ is diffeomorphic to a sphere. Suppose that the potential functions $\psi_S$ and $\psi_L$ where $d\psi_S=\lambda|_S$ and $d\psi_L =\lambda|_L$ are $C^0$-bounded by a small enough $\epsilon >0$, and consider the contact manifold $(Y',\xi')$ obtained by performing $(+1)$-surgery along $\Lambda^\epsilon_S$, an $\epsilon$-pushoff of $\Lambda$ by $\psi$, and $(-1)$-surgery along $\Lambda^{5\epsilon}_S$. Then:
\begin{enumerate}
\item there exists a canonical contact identification $(Y,\xi) = (Y',\xi')$ and
\item the Legendrian $\Lambda^{3\epsilon}_L\subset(Y',\xi')$ is Legendrian isotopic to $\Lambda^0_{\tau_S(L)}\subset (Y,\xi)$.
\end{enumerate}
In an analogous manner, performing contact $(-1)$ and $(+1)$-surgeries along $\Lambda^\epsilon_S$ and $\Lambda^{5\epsilon}_S$ in $(Y,\xi)$ results in a contact manifold $(Y,\xi)$ with a contact identification $(Y,\xi)=(Y',\xi')$ under which $\Lambda^{3\epsilon}_L\subset (Y,\xi)$ is Legendrian isotopic to $\Lambda_{\tau^{-1}_S(L)}\subset (Y,\xi)$.
\end{prop}

The recipe of \cite{CasMur19} is designed to find the attaching spheres in the Weinstein diagram by finding the lifts of corresponding vanishing cycles consisting of Dehn twists about some known spheres. In their paper, these vanishing cycles are obtained from a bifibration. For our Lefschetz fibrations, we already have the vanishing cycles in terms of Dehn twists, so we skip the steps of the recipe which construct the bifibration. We list the relevant steps here.

Let $\pi:(X,\omega)\to D^2$ be a Weinstein Lefschetz fibration with generic fiber $F$ and vanishing cycles $C_1,\dots,C_k$.

\begin{enumerate}
	\item Choose a set of exact Lagrangian spheres $\mathbb{L} = \{L_1 ,\dots,L_r \}$ in the generic fiber $F$ for which we understand the Legendrian lifts in the front projection of the contact boundary $\partial(F \times D^2)$. 
	\item Express each $\tau_{C_i}$ as a word in Dehn twists about the Lagrangian spheres in the set $\mathbb{L}$.
	\item For each vanishing cycle $C_i$, we apply Proposition \ref{prop:CM} to draw the front projection of their Legendrian lifts $\Lambda_i\subset \partial(F \times D^2)$.
	\item Then we consider the Legendrian link $\bigcup_i \Lambda_i$ determined by the cyclic ordering of the indices $i$: we push the Legendrian component $\Lambda_i$ in the Reeb direction by height equal to its index $i$, and this gives a well-defined link.
	\item Simplify the Weinstein handlebody diagram by applying the Legendrian Reidemeister moves, Gompf moves \cite{Gom98} (see Figure \ref{Gompf456}), handleslides \cite{DinGei04} (See Figure \ref{conecusp}), and handle cancellations. 
\end{enumerate}

\begin{figure}
	\begin{center}
		\includegraphics[width=12cm]{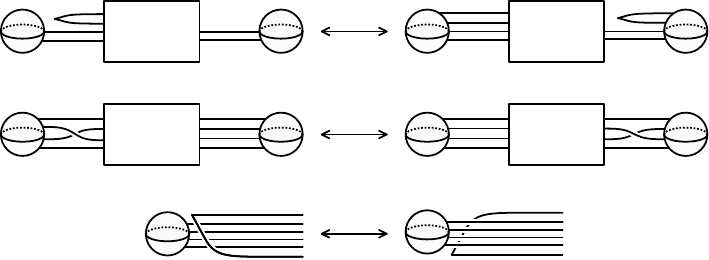}
		\caption{Gompf's three isotopic moves, up to 180 degree rotation about each axis. We call these moves Gompf 4, 5, and 6 respectively.}
		\label{Gompf456}
	\end{center}
\end{figure}

\begin{figure}
	\begin{center}
		\includegraphics[width=10 cm]{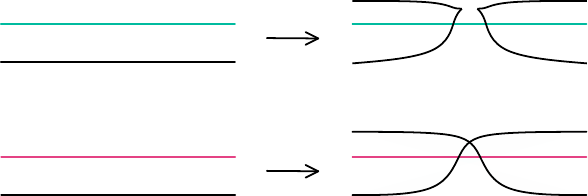}
		\caption[Handle slides in surgery diagrams.]{Top: handle sliding over a $(-1)$-surgery Legendrian (turquoise). Bottom: handle sliding over a $(+1)$-surgery Legendrian (pink) resulting in a cusp and a crossing respectively. The crossing may appear on either the top or bottom of the pink strand by a single application of Reidemeister 3.}
		\label{conecusp}
	\end{center}
\end{figure}

We apply this recipe to obtain the general form of the Weinstein diagram of a filling $X_\Lambda$ of $\Sigma_2(\Lambda')$:

\newtheorem*{thm:diagram}{Theorem \ref{thm:diagram}}
\begin{thm:diagram}
	Let $\Lambda\neq U$ be a Legendrian knot which is smoothly the closure of a quasipositive 3-braid of algebraic length 2. Let $\Lambda'$ be a positive transverse push off of $\Lambda$. Then there is a filling of $\Sigma_2(\Lambda')$, the double cover of $S^3$ branched over $\Lambda'$, given by the handle decomposition consisting of a single 1-handle and a single 2-handle pictured in Figure \ref{fig:d_Lambdaagain}.
\end{thm:diagram}

\begin{figure}
	\begin{center}
		\begin{tikzpicture}
			\node[inner sep=0] at (0,0) {\includegraphics[width=5 cm]{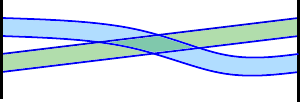}};
			\node at (-3, 0.4){$p$};
			\node at (-3.3, -0.3){$q-p$};
		\end{tikzpicture}
	\end{center}
	\caption[A Weinstein filling of $\Sigma_2(\Lambda')$.]{A Weinstein filling of $\Sigma_2(\Lambda')$. The shaded blue and shaded green represent $p$ and $q-p$ parallel strands respectively of a single Legendrian link in $S^1\times S^2$ where $0<p<q$ and $(p,q)=1$.}
	\label{fig:d_Lambdaagain}
\end{figure}

\begin{proof}
Let $\Lambda'$ be a transverse knot which is the closure of a quasipositive 3-braid of algebraic length 2. Then by Corollary \ref{cor:WeinsteinLF}, the double cover of $S^3$ branched over $\Lambda'$, $(\Sigma_2(\Lambda'),\xi)$ has Weinstein filling $X_\Lambda$. $X_\Lambda$ has a Weinstein Lefschetz fibration $\pi:X\rightarrow D^2$ with generic fiber $F$ a torus with one boundary component, and monodromy given by vanishing cycles which are Dehn twists along the curves $\alpha$ and $\gamma$. $\alpha$ is the $(1,0)$ curve and $\gamma$ is some $(p,q)$ curve in $F$. We apply Lemma \ref{lemma:qgreaterthanp} to ensure that the $(p,q)$ curve satisfies $0<p<q$.

\begin{figure}
	\begin{center}
		\includegraphics[width=12cm]{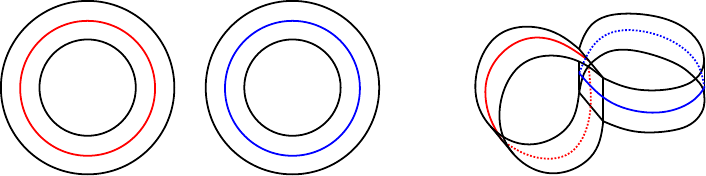}
		\caption[The Murasugi sum of two annuli paged open books.]{The annuli $A_1$ and $A_2$, and the fiber of the open book decomposition $(F_0,\phi_0)$ given by the Murasugi sum $(A_1,\tau_\alpha)*(A_2,\tau_\beta)$. The curve $\alpha$ is in red and the curve $\beta$ is in blue.}
		\label{fig:annuli}
	\end{center}
\end{figure}

\begin{figure}
	\begin{center}
		\begin{tikzpicture}
			\node[inner sep=0] at (0,0) {\includegraphics[width=12cm]{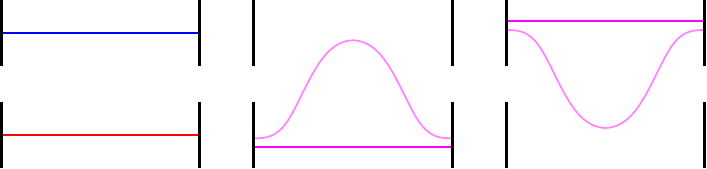}};
			\node at (-6.3,-1){$\alpha$};
			\node at (-6.3,1){$\beta$};
			\node at (1.2,-1.5){$+1$};
			\node at (1.2, 0.5){$-1$};
			\node at (5.5, 1.5){$+1$};
			\node at (5.5, -0.5){$-1$};
			\node at (-4.2, -1.9){\textbf{(1.)}};
			\node at (0, -1.9){\textbf{(2.)}};
			\node at (4.2, -1.9){\textbf{(3.)}};
		\end{tikzpicture}
		\caption[Surgeries corresponding to Dehn twists.]{From left to right:\\\hspace{\textwidth}
			\textbf{(1.)} the Legendrian lifts $\Lambda_\alpha$ and $\Lambda_\beta$ of $\alpha$ and $\beta$,\\\hspace{\textwidth}
			\textbf{(2.)} The $(+1)$ and $(-1)$ surgery curves for a Dehn twist $\tau_\alpha$, 	\\\hspace{\textwidth}
			\textbf{(3.)} The $(+1)$ and $(-1)$ surgery curves for a Dehn twist  $\tau_\beta^{-1}$.}
		\label{fig:skeleton}
	\end{center}
\end{figure}

We now apply the recipe of Casals and Murphy. \textbf{Step 1:} We need a set of Lagrangian spheres for which we understand the Legendrian lifts. We choose the curves $\alpha$ and $\beta$ of slopes $(1,0)$ and $(0,1)$ respectively in $F$, see Figure \ref{fig:curvesontorus}.  To find the Legendrian lifts $\Lambda_\alpha$ and $\Lambda_\beta$ of $\alpha$ and $\beta$, we consider the following. Let $(F_0,\phi_0)$ be the open book decomposition given by $F_0\defeq F$ and $\phi_0$ given by Dehn twists $\tau_\alpha$ and $\tau_\beta$. This open book decomposition is the Murasugi sum $(A_1,\tau_\alpha)*(A_2,\tau_\beta)$ where $A_1, A_2$ are annuli with monodromy given by Dehn twists about their cores, $\alpha$ and $\beta$, see Figure \ref{fig:annuli}. Thus,
$$(F_0,\phi_0)=(S^3\#S^3,\xi_{std}\#\xi_{std})= (S^3,\xi_{std})$$
and the Legendrian lifts $\Lambda_\alpha$ and $\Lambda_\beta$ of $\alpha$ and $\beta$ are the lifts of the cores of the annuli $A_1$ and $A_2$. $\Lambda_\alpha$ and $\Lambda_\beta$ each correspond to the attaching sphere of a critical 2-handle cancelling a subcritical 1-handle attachment, as in Figure \ref{fig:skeleton}. 

\textbf{Step 2} is to express $\alpha$ and $\gamma$ as the images of $\alpha$ under words in Dehn twists along $\alpha$ and $\beta$. $\alpha$ is as given. For $\gamma$ it suffices to find a series of Dehn twists about $\alpha$ and $\beta$ on the curve $(1,0)$ to obtain the curve with slope $(p,q)$ in the torus with one boundary component. Following Table \ref{table:pq}, $(p,q)$ is obtained by some sequence of the Dehn twists $\tau_\alpha$ and $\tau_{\beta}^{-1}$. As in Remark \ref{remark:TorusDtwistcombinatorics}, we perform the Euclidean algorithm to obtain successive coefficients $n_i\in\ZZ^+$:
\begin{align*}
	q&=n_kp+r_k\\
	p&=n_{k-1}r_k+r_{k-1}\\
	r_k&=n_{k-2}r_{k-1}+r_{k-2}\\
	&\dots\\
	r_3&=n_3r_2+r_1\\
	r_2&=n_2r_1+1\\
	r_1&=n_1.
\end{align*}
Then,
$$\tau_{\beta}^{-n_k}\circ\tau_\alpha^{n_{k-1}}\circ\tau_{\beta}^{-n_{k-2}}\circ\dots\circ\tau_{\beta}^{-n_1}(\alpha) = \gamma.$$

\textbf{Step 3} is to use Proposition \ref{prop:CM} repeatedly to obtain the Legendrian lifts $\Lambda_\alpha$ and $\Lambda_\gamma$. $\Lambda_\alpha$ is given and consists of an unknotted Legendrian knot which cancels the 1-handle labelled $\alpha$ in Figure \ref{fig:skeleton}. 

To find $\Lambda_\gamma$, we draw the curve
$$\Lambda_{\tau_{\beta}^{-n_k}\circ\tau_\alpha^{n_{k-1}}\circ\tau_{\beta}^{-n_{k-2}}\circ\dots\circ\tau_{\beta}^{-n_1}(\alpha)}.$$

\begin{figure}
	\begin{center}
		\begin{tikzpicture}
			\node[inner sep=0] at (0,0) {\includegraphics[width=12cm]{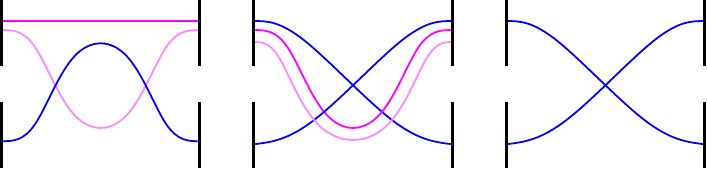}};
			\node at (-4.5, 1.4){$+1$};
			\node at (-4.5,-1.1){$-1$};
			\node at (-4.2, -1.9){\textbf{(1.)}};
			\node at (0, -1.9){\textbf{(2.)}};
			\node at (4.2, -1.9){\textbf{(3.)}};
		\end{tikzpicture}
		\caption[How to draw the Legendrian lift of $\tau_{\beta}^{-1}(\alpha)$.]{The Legendrian lift of $\tau_{\beta}^{-1}(\alpha)$. From left to right:\\\hspace{\textwidth}
			\textbf{(1.)} First the diagram given by Proposition \ref{prop:CM}, \\\hspace{\textwidth}
			\textbf{(2.)} we perform a handle slide of the blue curve over the $(+1)$ curve, 
			\\\hspace{\textwidth}
			\textbf{(3.)} and then we cancel out the parallel $(+1)$ and $(-1)$ surgery curves.}
		\label{fig:pq_1}
	\end{center}
\end{figure}

\begin{figure}
	\begin{center}
		\begin{tikzpicture}
			\node[inner sep=0] at (0,0) {\includegraphics[width=14cm]{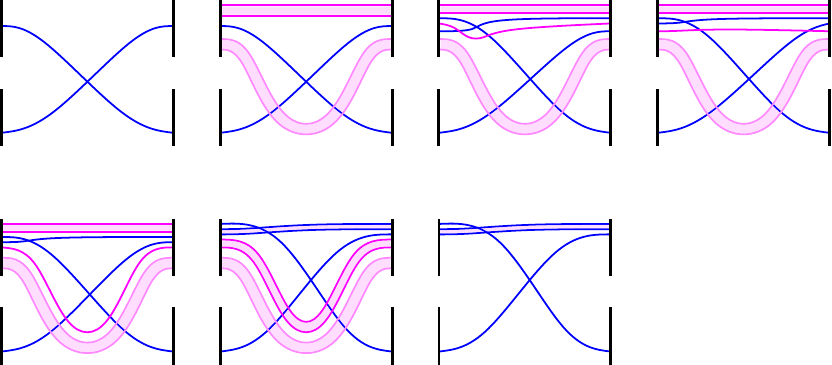}};
			\node at (-2.5, 3.2){$+1$};
			\node at (-2.5, 0.7){$-1$};
			\node at (-5.5, 0.2){\textbf{(1.)}};
			\node at (-1.85, 0.2){\textbf{(2.)}};
			\node at (1.85, 0.2){\textbf{(3.)}};
			\node at (5.5, 0.2){\textbf{(4.)}};
			\node at (-5.5, -3.4){\textbf{(5.)}};
			\node at (-1.85, -3.4){\textbf{(6.)}};
			\node at (1.85, -3.4){\textbf{(7.)}};
		\end{tikzpicture}
		\caption{How to draw the Legendrian lift of $\tau_{\beta}^{-n_1}(\alpha)$. Shaded pink ribbons represent $n_1-1$ parallel $(+1)$ and $(-1)$ curves given by Proposition \ref{prop:CM}. Describing these 7 diagrams in order: \\\hspace{\textwidth}
			\textbf{(1.)} $\tau_{\beta}^{-1}(\alpha)$, \\\hspace{\textwidth}
			\textbf{(2.)} $n_1-1$ parallel $(+1)$ and $(-1)$ curves added in, \\\hspace{\textwidth}
			\textbf{(3.)} we slide the blue curve over the lowest of the parallel $(+1)$ curves (the topmost pink ribbon now represents $n_1-2$ parallel strands), \\\hspace{\textwidth}
			\textbf{(4.)} Gompf 5, \\\hspace{\textwidth}
			\textbf{(5.)} Reidemeister 3, \\\hspace{\textwidth}
			\textbf{(6.)} Repeat steps 3-6 with remaining $(+1)$ curves, and \\\hspace{\textwidth}
			\textbf{(7.)} we cancel the $(+1)$ and $(-1)$ curves to obtain $\tau_{\beta}^{-n_1}(\alpha)$.}
		\label{fig:pq_2}
	\end{center}
\end{figure}

We apply Proposition \ref{prop:CM} repeatedly to obtain the surgery diagram. This means repeatedly adding in a $(+1)$ and a $(-1)$ surgery curve along either $\alpha$ or $\beta$ with heights determined by Proposition \ref{prop:CM}, see Figure \ref{fig:skeleton}, and sliding over the current handle to cancel them out. We see that applying $\tau_\beta^{-1}$ increases the number of strands going through the topmost 1-handle labelled $\beta$ in Figure \ref{fig:skeleton}, and applying $\tau_\alpha$ increases the number of strands going through the bottom 1-handle labelled $\alpha$ in Figure \ref{fig:skeleton}. 

To begin, the lift of $\tau_{\beta}^{-1}(\alpha)$ is given by Figure \ref{fig:pq_1}. Next, assuming $n_1 > 1$, applying more twists $\tau_{\beta}^{-n_1-1}$ results in the series of diagrams in Figure \ref{fig:pq_2}. In $\tau_{\beta}^{-n_1}(\alpha)$, there are now $n_1$ strands going through the 1-handle labelled $\beta$ and 1 strand going through the 1-handle labelled $\alpha$.  

\begin{figure}
	\begin{center}
		\begin{tikzpicture}
			\node[inner sep=0] at (0,0) {\includegraphics[width=14cm]{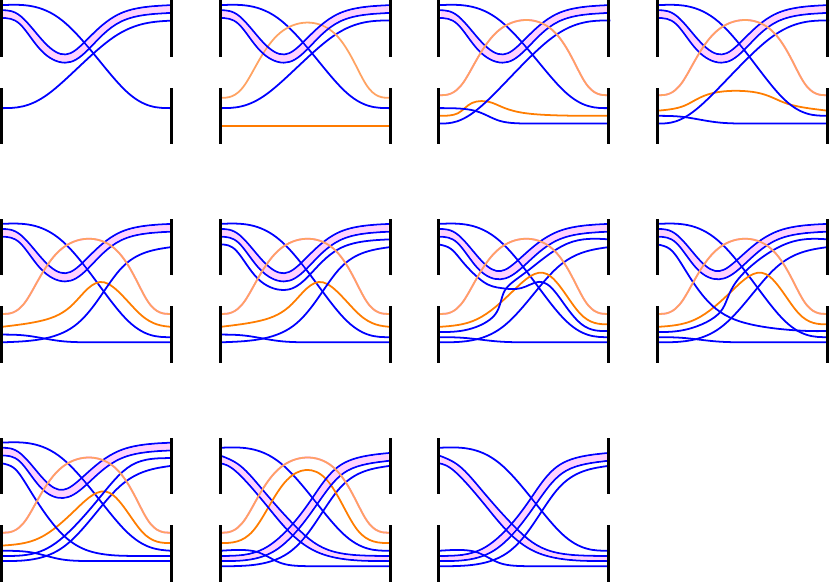}};
			\node at (-1.7, 4.75){$-1$};
			\node at (-1, 2.55){$+1$};
			\node at (-5.5, 2.1){\textbf{(1.)}};
			\node at (-1.85, 2.1){\textbf{(2.)}};
			\node at (1.85, 2.1){\textbf{(3.)}};
			\node at (5.5, 2.1){\textbf{(4.)}};
			\node at (-5.5, -1.6){\textbf{(5.)}};
			\node at (-1.85, -1.6){\textbf{(6.)}};
			\node at (1.85, -1.6){\textbf{(7.)}};
			\node at (5.5, -1.6){\textbf{(8.)}};
			\node at (-5.5, -5.4){\textbf{(9.)}};
			\node at (-1.85, -5.4){\textbf{(10.)}};
			\node at (1.85, -5.4){\textbf{(11.)}};
		\end{tikzpicture}
		\caption{How to draw the Legendrian lift of $\tau_\alpha\tau_{\beta}^{-n_1}(\alpha)$. Shaded pink ribbons represent $n_1-1$ parallel curves. Describing these 11 diagrams in order: \\\hspace{\textwidth}
			\textbf{(1.)} $\tau_{\beta}^{-n_1}(\alpha)$, \\\hspace{\textwidth}
			\textbf{(2.)} a $(+1)$ and a $(-1)$ curve added in, \\\hspace{\textwidth}
			\textbf{(3.)} we slide the blue curve over the orange $(+1)$ curve, \\\hspace{\textwidth}
			\textbf{(4.)} Gompf 5, \\\hspace{\textwidth}
			\textbf{(5.)} Reidemeister 3, \\\hspace{\textwidth}
			\textbf{(6.)} we separate the bottom most blue curve from the pink ribbon (which now represents $n_1-2$ parallel strands), \\\hspace{\textwidth}
			\textbf{(7.)} we slide the blue curve over the orange, \\\hspace{\textwidth}
			\textbf{(8.)} Reidemeister 3, \\\hspace{\textwidth}
			\textbf{(9.)} Reidemeister 3, \\\hspace{\textwidth}
			\textbf{(10.)} repeat steps 7-9 with the remaining strands in the pink ribbon, \\\hspace{\textwidth}
			\textbf{(11.)} we cancel the $(+1)$ and $(-1)$ orange curves to obtain $\tau_\alpha\tau_{\beta}^{-n_1}(\alpha)$.}
		\label{fig:pq_3}
	\end{center}
\end{figure}

We now apply $\tau_\alpha$ to $\tau_{\beta}^{-n_1}(\alpha)$, resulting in the series of diagrams in Figure \ref{fig:pq_3}. In $\tau_\alpha\tau_{\beta}^{-n_1}(\alpha)$, there are now $n_1$ strands going through the 1-handle labelled $\beta$ and $n_1+1$ strands going through the 1-handle labelled $\alpha$. Now assuming $n_2>1$, we apply $\tau_\alpha^{n_2-1}$ to $\tau_\alpha\tau_{\beta}^{-n_1}(\alpha)$. We obtain the diagrams in Figure \ref{fig:pq_4}. In the last diagram showing $\tau_\alpha^{n_2}\tau_{\beta}^{-n_1}(\alpha)$, we can count $n_1$ strands going through the handle labelled $\beta$ and $n_1n_2+1$ strands going through the handle labelled $\alpha$.

\begin{figure}
	\begin{center}
		\begin{tikzpicture}
			\node[inner sep=0] at (0,0) {\includegraphics[width=14cm]{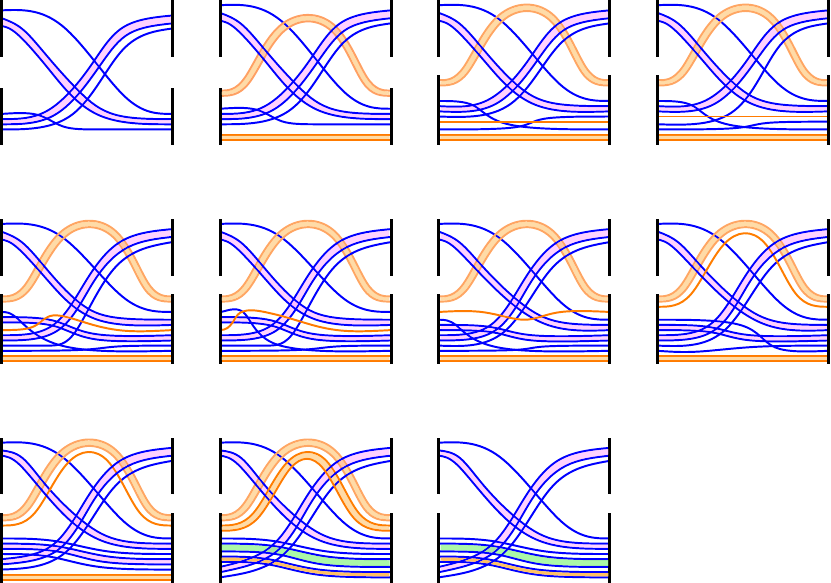}};
			\node at (-1.5, 4.8){$-1$};
			\node at (-1, 2.3){$+1$};
			\node at (-5.5, 2.1){\textbf{(1.)}};
			\node at (-1.85, 2.1){\textbf{(2.)}};
			\node at (1.85, 2.1){\textbf{(3.)}};
			\node at (5.5, 2.1){\textbf{(4.)}};
			\node at (-5.5, -1.6){\textbf{(5.)}};
			\node at (-1.85, -1.6){\textbf{(6.)}};
			\node at (1.85, -1.6){\textbf{(7.)}};
			\node at (5.5, -1.6){\textbf{(8.)}};
			\node at (-5.5, -5.4){\textbf{(9.)}};
			\node at (-1.85, -5.4){\textbf{(10.)}};
			\node at (1.85, -5.4){\textbf{(11.)}};
		\end{tikzpicture}
		\caption[Drawing $\tau_\alpha^{n_2}\tau_{\beta}^{-n_1}(\alpha)$.]{Shaded pink ribbons represent $n_1-1$ parallel curves, shaded orange regions represent $n_2-1$ parallel curves, shaded green region represents $(n_1-1)(n_2-1)$ curves. Describing these 11 diagrams in order:\\\hspace{\textwidth}
			\textbf{(1.)} $\tau_\alpha\tau_{\beta}^{-n_1}(\alpha)$, \\\hspace{\textwidth}
			\textbf{(2.)} $n_2-1$ $(+1)$ and $(-1)$ curves added in, \\\hspace{\textwidth}
			\textbf{(3.)} slide the blue curve over the topmost of the parallel $(+1)$ curves, \\\hspace{\textwidth}
			\textbf{(4.)} Gompf 5, \\\hspace{\textwidth}
			\textbf{(5.)} $n_1-1$ handleslides of the blue curves in the pink ribbon over the orange $(+1)$ curve, \\\hspace{\textwidth}
			\textbf{(6.)} repeated applications of Reidemeister 3, \\\hspace{\textwidth} 
			\textbf{(7.)} repeated applications of Gompf 5, \\\hspace{\textwidth}
			\textbf{(8.)} repeated applications of Reidemeister 3, \\\hspace{\textwidth}
			\textbf{(9.)} repeated applications of Gompf 5, \\\hspace{\textwidth}
			\textbf{(10.)} repeat steps 3-9 with the remaining strands in the orange ribbon, \\\hspace{\textwidth}
			\textbf{(11.)} we cancel the $(+1)$ and $(-1)$ curves to obtain $\tau_\alpha^{n_2}\tau_{\beta}^{-n_1}(\alpha)$.}
		\label{fig:pq_4}
	\end{center}
\end{figure}

We continue to apply successive $\tau_{\beta}^{-n_i}$ and $\tau_\alpha^{n_j}$ twists to the curve. In doing so, we obtain the diagrams of Figure \ref{fig:pq_5}. Applying $\tau_{\beta}^{-n_i}$ results in the top row of diagrams, and applying $\tau_\alpha^{n_j}$  results in the bottom row of diagrams. The count of strands passing through the 1-handles after each successive step correspond to the $r_i$ in the Euclidean algorithm.

\begin{figure}
	\begin{center}
		\begin{tikzpicture}
			\node[inner sep=0] at (0,0) {\includegraphics[width=12cm]{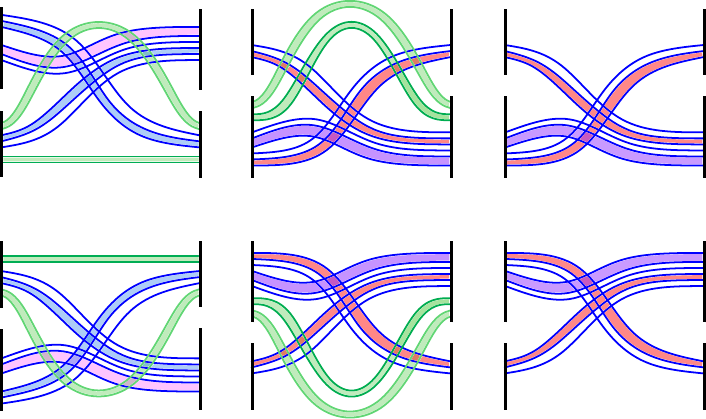}};
			\node at (-4, 3.5){$-1$};
			\node at (-4, 0.5){$+1$};
			\node at (-4, -3.5){$-1$};
			\node at (-4, -0.5){$+1$};
		\end{tikzpicture}
		\caption[Effect of applying $\tau_\alpha^{n_j}$ and $\tau_{\beta}^{-n_i}$.]{Coloured ribbons represent some number of parallel strands. In the top row, we apply Dehn twists $\tau_\alpha^{n_j}$ and see the resulting diagram (the step by step procedure follows from Figure \ref{fig:pq_4}). In the bottom row, we apply Dehn twists $\tau_{\beta}^{-n_i}$, and see the resulting diagram (the step by step procedure consists of the same diagrams as the top row, reflected about a horizontal axis).}
		\label{fig:pq_5}
	\end{center}
\end{figure}

In the case that $n_1 = 1$, the diagrams in Figure \ref{fig:pq_2} reflected horizontally give the $\tau_\alpha^{n_{2}}\tau_{\beta}^{-1}(\alpha)$, reflected versions of Figures \ref{fig:pq_3} and \ref{fig:pq_4} give $\tau_{\beta}^{-1}\tau_\alpha^{n_{2}}\tau_{\beta}^{-1}(\alpha)$ and $\tau_{\beta}^{-n_3}\tau_\alpha^{n_{2}}\tau_{\beta}^{-1}(\alpha)$ respectively. In general, we still end up with the diagrams of Figure $\ref{fig:pq_5}$, with every additional application of  $\tau_\alpha^{n_j}$ resulting in the top row of diagrams, and every additional application of  $\tau_{\beta}^{-n_i}$ resulting in the bottom row of diagrams.

Since $p<q$, the final Dehn twist applied will be a $\tau_{\beta}^{-1}$, so the final diagram will correspond to the bottom right of Figure \ref{fig:pq_5}, with $q$ strands going through the top 1-handle and $p$ strands going through the bottom 1-handle. We have found a Legendrian representative for $\Lambda_\gamma$.

We proceed with \textbf{Step 4}. Since we are dealing with only 2 curves, the lifts of $\alpha$ and $\gamma$, we choose the cyclic ordering that places $\Lambda_\gamma$ above $\Lambda_\alpha$, giving us the top left diagram of Figure \ref{fig:pq_final}.

\begin{figure}
	\begin{center}
		\includegraphics[width=8cm]{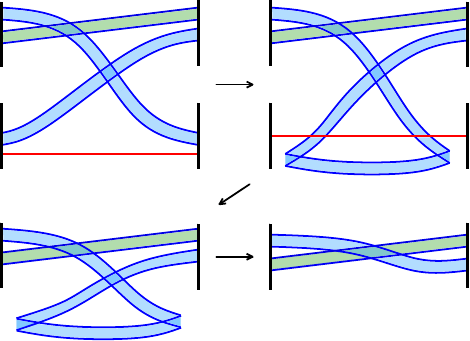}
		\caption[Simplifying the Weinstein diagram.]{The blue shaded region represents $p$ parallel curves and the green shaded region represents $q-p$ parallel curves. 1. The Weinstein diagram obtained after Step 4 of the recipe, 2. handleslide the blue curves over the red curve, 3. cancel the bottom 1-handle with the red curve, 4. perform a series of Reidemeister moves on the parallel blue curves.}
		\label{fig:pq_final}
	\end{center}
\end{figure}

\textbf{Step 5} is to simplify the handle diagram. We do so following Figure \ref{fig:pq_final} by some knot isotopies, a handle slide, and cancelling the bottom 1-handle with $\Lambda_\alpha$. The resulting diagram has one 2-handle winding about a single 1-handle $q$ times. We can see this in the attaching curve in the bottom right of Figure $\ref{fig:pq_final}$ which is colour coded: the blue shaded region represents $p$ parallel curves and the green shaded region represents $q-p$ parallel curves.
\end{proof}

We will call the Weinstein diagram in Gompf standard form obtained in Theorem \ref{thm:diagram} method $\mathcal{D}(\Lambda)$ for the knot $\Lambda$, where $\mathcal{D}(\Lambda)$ depicts a filling of $\Sigma_2(\Lambda')$. 

\begin{figure}
	\begin{center}
		\begin{tikzpicture}
			\node[inner sep=0] at (0,0) {\includegraphics[width=12cm]{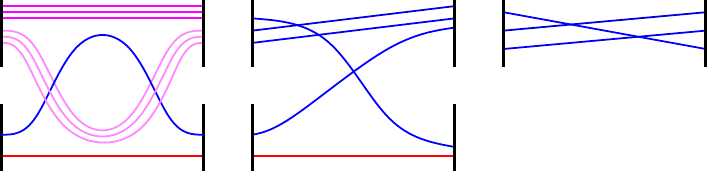}};
			\node at (-4, 1.6){$+1$};
			\node at (-4, -0.25){$-1$};
		\end{tikzpicture}
		\caption[A filling of the double cover of the transverse $m(8_{20})$ knot.]{$\mathcal{D}(m(8_{20}))$, the Weinstein diagram for the filling of the double cover of the transverse $m(8_{20})$ knot, following Theorem \ref{thm:diagram}.}
		\label{fig:8_20WD}
	\end{center}
\end{figure}

\begin{eg}\label{eg:8_20diagram}
To illustrate the proof of the previous theorem, we will apply the Weinstein handlebody diagram drawing procedure to a particular knot.

The transverse $m(8_{20})$ knot \cite{knotatlas} is the closure of the braid $\sigma_1\sigma_2^3\sigma_1\sigma_2^{-3}$ \cite{SnapPy}. We conjugate this braid word by $\sigma_1\sigma_2^{-3}$ and obtain an equivalent braid:
$$\sigma_1\sigma_2^3\sigma_1\sigma_2^{-3}=\sigma_1\sigma_2^{-3}(\sigma_1\sigma_2^3\sigma_1\sigma_2^{-3})\sigma_2^3\sigma_1^{-1} = \sigma_1\sigma_2^{-3}\sigma_1\sigma_2^3.$$

By Corollary \ref{cor:WeinsteinLF}, the double cover of $S^3$ branched over this knot has a Weinstein Lefschetz fibration $\pi:X\rightarrow D^2$ with generic fibre $F$, a torus with one boundary component containing vanishing cycles given by $\alpha$ and $\tau_{\beta}^{-3}(\alpha)$ which are curves of slope $(1,0)$ and $(1,3)$ respectively.

Following the steps in the proof of Theorem \ref{thm:diagram}, we obtain the lift of the $(1,3)$ curve and the surgery diagram $\mathcal{D}(m(8_{20}))$, as in Figure \ref{fig:8_20WD}.
\end{eg}

\begin{remark}\label{remark:more_strands_issues}
	This strategy can be used to obtain the Weinstein diagram $X_\Lambda$ for a filling of $\Sigma_2(\Lambda')$ where $\Lambda'$ is a quasipositive $n$-braid of algebraic length $n-1$ (note that other $n$-braids are obstructed from the Lagrangian double concordance by Corollary \ref{cor:n_braid}). However, the generic fiber of the Lefschetz fibration of $X_\Lambda$ for an $n$-stranded braid will have an additional 1-handle for each additional strand. Because of this added complexity, the class of possible vanishing cycles cannot be classified as easily as in the 3-braid case. Furthermore, the Weinstein diagram will involve the attachment of $n-1$ 2-handles. Thus, while it follows naturally to generate the Weinstein diagram of an $n$-braid in specific examples, it is more difficult to generate the Weinstein diagram for a general $n$-braid.
\end{remark}

\begin{eg}\label{eg:m10_140}
	We can for instance let $\Lambda$ have the same smooth type as the mirror of the $10_{140}$ knot which is given by the quasipositive $4$-stranded braid $\sigma_1^{-3}\sigma_2\sigma_1^3\sigma_2\sigma_3\sigma_2^{-1}\sigma_3$ according to \cite{knotatlas}. The braid has algebraic length 3. The Legendrian contact homology DGA of $m(10_{140})$ is stable tame isomorphic to the DGA of the standard unknot, and thus Lagrangian concordance from $\Lambda$ to the unknot cannot be obstructed by contact homology techniques. We use the construction from this chapter to construct a familiar Weinstein handlebody diagram of a filling of $\Sigma_2(\Lambda')$, so that it follows from the work of the subsequent sections of this paper that $m(10_{140})$ cannot be concordant to the unknot.
	
	\begin{figure}
		\begin{center}
			\begin{tikzpicture}
				\node[inner sep=0] at (0,0) {\includegraphics[width=10cm]{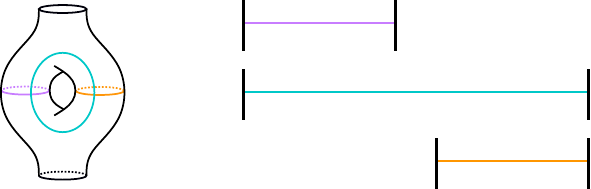}};
				\node at (-5.3, 0){$\alpha$};
				\node at (-4, -0.9){$\beta$};
				\node at (-2.55, 0){$\gamma$};
				\node at (2, 1.2){$\alpha$};
				\node at (5.3, 0){$\beta$};
				\node at (5.3, -1.2){$\gamma$};
			\end{tikzpicture}
			\caption{On the left, the fiber of the Weinstein Lefschetz fibration of a filling of the double cover of the $m(10_{140})$ knot and a basis of Lagrangian spheres $\alpha$, $\beta$, and $\gamma$. On the right, the Legendrian lifts of these spheres to a Weinstein diagram.}
			\label{fig:m10140_fiber}
		\end{center}
	\end{figure}
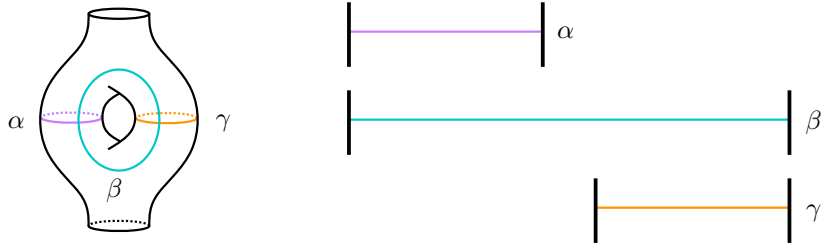
	
	\begin{figure}
		\begin{center}
			\includegraphics[width=15cm]{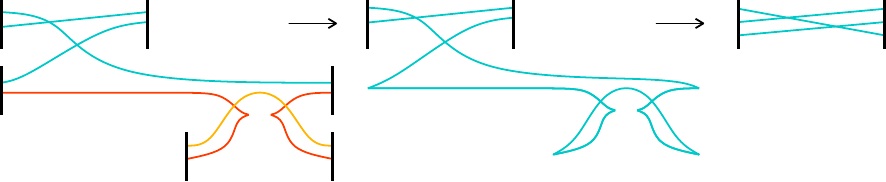}
			\caption{A Weinstein handlebody diagram corresponding to a filling of the double cover of $m(10_{140})$. The Legendrian lifts of $\tau_\beta(\gamma)$, $\gamma$, and $\tau_\alpha^{-3}(\beta)$ are arranged in increasing Reeb height in orange, red, and turquoise respectively. We then perform a series of handle slides, handle cancellations, and Reidemeister moves.}
			\label{fig:m10140_2handles}
			\label{fig:m10140_calc}
		\end{center}
	\end{figure}
	
	First we perform a cyclic conjugation to work with the equivalent braid word: $\sigma_2\sigma_3\sigma_2^{-1}\sigma_3\sigma_1^{-3}\sigma_2\sigma_1^3$. Then an open book decomposition of $\Sigma_2(\Lambda')$ has pages consisting of the torus with two boundary components, and monodromy given by Dehn twists about the curves $\tau_\beta(\gamma)$, $\gamma$, and then $\tau_\alpha^{-3}(\beta)$, where $\alpha$, $\beta$, and $\gamma$ are the curves in on the fiber shown in Figure \ref{fig:m10140_fiber}. We attach handles along these curves to obtain a Weinstein Lefschetz fibration of a filling of $\Sigma_2(\Lambda')$. The Legendrian lifts $\alpha$, $\beta$, and $\gamma$ are also pictured in Figure \ref{fig:m10140_fiber}. The Legendrian lifts of $\tau_\beta(\gamma)$, $\gamma$, and then $\tau_\alpha^{-3}(\beta)$ given by Proposition \ref{prop:CM} are arranged with their relative Reeb heights given by their order in the monodromy in Figure \ref{fig:m10140_calc}. This is a Weinstein handlebody diagram corresponding to the filling. We simplify the diagram to obtain a familiar Weinstein diagram, seen previously in Example \ref{eg:8_20diagram}. 	
\end{eg}

%% file: Chapter5.tex
\section{The Legendrian contact homology DGA}\label{sec:DGA}

In this section, we compute the Legendrian contact homology DGA of the knots in the diagrams $\mathcal{D}(\Lambda)$ of the previous section, drawn in $S^1\times S^2$. Let $L$ be the knot depicted in the diagram $\mathcal{D}(\Lambda)$. We'll denote its Legendrian contact homology DGA as $(\mathscr{A}_L,\partial_L)$.

For details on how to compute the Legendrian contact homology DGA for Legendrian knots, see \cite{EtnNg18}. In our situation, since $L$ winds about a 1-handle, we follow the process in \cite{EkhNg15} and the simplification in \cite{EtgLek19}. We now briefly summarize this procedure. We begin by taking the Lagrangian resolution of $L$ following \cite{Ng03} where the Lagrangian projection of $L$ is depicted with additional 1-handle half twist from Definition 2.3 of \cite{EkhNg15} (See for instance, Figure \ref{fig:8_20LagRes}). The DGA is the associative, noncommutative, unital algebra over the base ring $\ZZ[t, t^{-1}]e_1$ where $e_1$ is an idempotent element, generated by Reeb chords of the Legendrian knot $\{a_k,c^0_{i,j}, c^1_{i,j}\}$. The $a_k$ generators correspond to crossings of the knot (called the external generators) and the $c^0_{i,j}, c^1_{i,j}$ generators correspond to a subset of the Reeb chords in the 1-handle (called the internal generators). The differentials of the external generators come from a count of certain immersed disks in $L$ with corners at Reeb chords of the crossings or the 1-handle. The differentials of the internal generators are 
\begin{align*}
	\partial(c^0_{i\;j}) &= \sum_{m=1}^n\sigma_i\sigma_m c^0_{i\;m} c^0_{m\;j}\\
	\partial(c^1_{i\;j}) &= \delta_{i\;j}+\sum_{m=1}^n\sigma_i\sigma_m c^0_{i\;m} c^1_{m\;j}+\sum_{m=1}^n\sigma_i\sigma_m c^1_{i\;m} c^0_{m\;j}\\
\end{align*}
where $\delta_{i\;j} = e_1$ if $i=j$ and 0 otherwise. We extend $\partial_L$ by the signed Leibniz rule. The grading comes from a chosen Maslov potential: a kind of locally constant map on the strands of $L$ that changes by 1 when traversing a cusp in the front diagram.

\subsection{Computing the DGA}\label{subsec:computeDGA}

We now explicitly compute $(\mathscr{A}_L,\partial_L)$. We label the crossings in the Lagrangian resolution of $L$ that were present in the front projection (ie. the crossings on the left side of the diagram) with $a_i$, with the left most crossing labelled $a$. We label crossings which come from the 1-handle twist of the Lagrangian resolution with $b_i$, so that the indices increase from left to right starting with the bottom row and moving upwards. We place a marked point at the minimum point of the top strand exiting the 1-handle on the left, and we label the marked point $t$. See the partially labelled Figure \ref{fig:GenDGA} and the fully labelled examples of the Lagrangian resolution of $\mathcal{D}(\Lambda)$, where $\Lambda$ is the $m(8_{20})$ knot, in Figure \ref{fig:8_20DGA}, and where $\Lambda$ is the $10_{155}$ knot, in Figure \ref{fig:10155lag}.

\begin{figure}
	\begin{center}
		\begin{tikzpicture}
			\node[inner sep=0] at (0,0) {\includegraphics[width=14 cm]{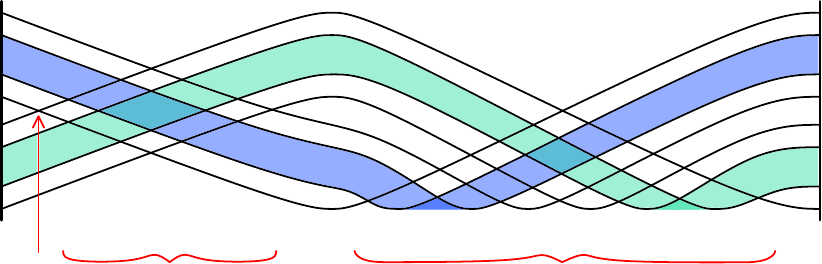}};
			\node at (-7.3,2){1};
			\node at (-7.3,1.5){\vdots};
			\node at (-7.3,0.6){$p$};
			\node at (-7.3,-0.5){\vdots};
			\node at (-7.3,-1.4){$q$};
			\node at (-6.3,-2.3){$a$};
			\node at (-4,-2.5){$a_i$};
			\node at (2.6,-2.5){$b_j$};
			\node at (-0.6,-1.5){$b_1$};
			\node at (0.4,-1.5){$\dots$};
			\node at (4,-1.5){$\dots$};
			\node at (2.5,-1.5){$b_p$};
			\node at (6,-1.5){$b_{q-1}$};
			\node at (2,-1.7){$t$};
			\node at (2,-1.35){$\bullet$};
		\end{tikzpicture}
		\caption[The Lagrangian projection of $\mathcal{D}(\Lambda)$ labelled for a DGA computation.]{$\mathcal{D}(\Lambda)$. There are $p-2$ parallel curves in the blue band, and $q-p-2$ curves in the green band. The $a$ and $b_i$ generators of the DGA are labelled. The other crossing are labelled $a_i$ if they are in the region on the left, and $b_j$ if they are in the region on the right.}
		\label{fig:GenDGA}
	\end{center}
\end{figure}
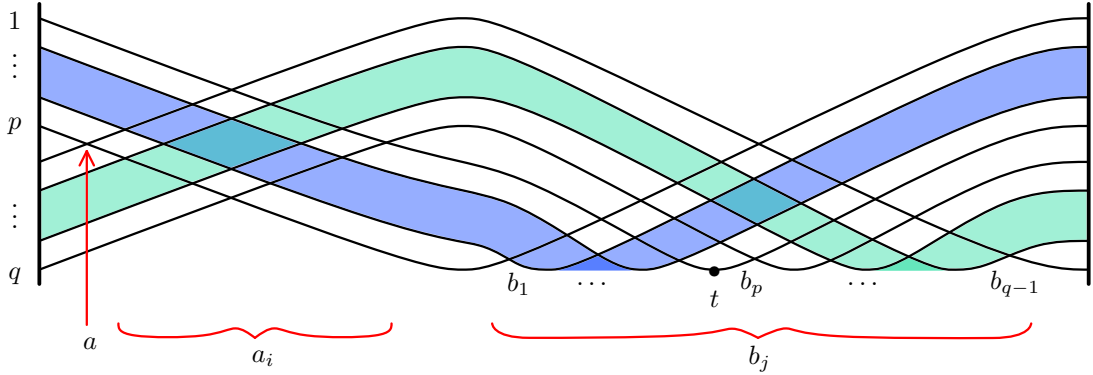

We can now compute the Legendrian contact homology DGA $(\mathscr{A}_L,\partial_L)$. Since there are no cusps, we assign each strand the Maslov potential 0. The generators are $a, a_1,\dots, a_k, b_1,\dots, b_l$, $c^0_{i\;j}$ for $1\leq i<j\leq q$, and $c^1_{i\;j}$ for $1\leq i,j\leq q$. The gradings of the generators are as follows:
\begin{align*}
	&|t|=|a|=|a_i|=|b_i|=0 && |c^0_{i\;j}|= -1 && |c^1_{i\;j}|= 1
\end{align*}

The differentials of generators of the internal DGA are:
\begin{align*}
	\partial(c^0_{i\;j}) &= \sum_{m=1}^qc^0_{i\;m} c^0_{m\;j}\\
	\partial(c^1_{i\;j}) &= \delta_{i\;j}+\sum_{m=1}^qc^0_{i\;m} c^1_{m\;j}+\sum_{m=1}^q c^1_{i\;m} c^0_{m\;j}
\end{align*}
where $\delta_{i\;j}=e_1$ if $i=j$ and is $0$ otherwise.

The differentials of the external DGA are given by the count of immersed disks. Notably, $\partial a$, from the leftmost crossing of Figure \ref{fig:GenDGA}, has one contributing disk to its left labelled $c^0_{p\;p+1}$ in the differential. Each of the differentials $\partial b_i$ for $i=\{1, \dots,q-1\}$ count two disks, one to the right of $b_i$ and one to the left, except $\partial b_p = c^0_{q-p\;q-p+1}$ which only counts one disk to the right. For such $i$, the marked point $t$ appears in the differential of a generator only when $p>1$ and then only once, in $\partial b_{p-1}$. Explicitly, the differential for $a$ and the $b_i$'s are as follows:
\begin{align*}
	\partial a &= c^0_{p\;p+1} \tag{\textasteriskcentered}\\
	\partial b_1 &= c^0_{q-1\;q}+c^0_{p-1\;p}\\
	\partial b_2 &= c^0_{q-2\;q-1}+ c^0_{p-2\;p-1}\\
	\dots& \\
	\partial b_{p-1}&= tc^0_{q-p+1\;q-p+2}+c^0_{1\;2}\\
	\partial b_p &= c^0_{q-p\;q-p+1}\\
	\partial b_{p+1} &= c^0_{q-p-1\;q-p} +c^0_{q-1\;q}\\
	\dots& \\
	\partial b_{q-1} &= c^0_{1\;2}+c^0_{p+1\;p+2}
\end{align*}
For $a_i$ and $b_j$ where $j > q-1$, the differential counts at least $2$ disks, including at least one which contributes a term of the form $\mu c^0_{i'\;j'}$ or $c^0_{i'\;j'}\mu$ where $\mu$ is $a$, or some $a_i$ or $b_j$.
We extend $\partial_L$ by the signed Leibniz rule. See Example \ref{eg:8_20DGA} and Example \ref{eg:10_155} for a full computation of the differential on the external generators for some particular $\mathcal{D}(\Lambda)$.

\begin{lemma}\label{lemma:nosolos}
Let $K$ be the Legendrian knot of a diagram $\mathcal{D}(\Lambda)$. Let $q$ be the number of strands passing through the 1-handle. Then the terms $a, b_1,\dots,b_{q-1}$ do not appear in any differentials of the Legendrian contact homology DGA $\mathscr{A}_K/\ZZ \langle e_1\rangle$ as degree 1 monomials.
\end{lemma}

\begin{proof}
Suppose $a$ appears in the differential of some generator $\mu$. Then there is an immersed disk with a negative corner at $a$. The boundary of this immersed disk must follow one of the strands from $a$ to the left. Suppose it's the overcrossing strand. Then the strand immediately enters the 1-handle, so $a$ appears in the differential in a word of the form $\mu_1 a c^0_{i\; p} \mu_2$ or  $\mu_1 c^0_{i\; p} a \mu_2$, where $i<p$ and $\mu_i$ is some other string of generators, possibly an empty one. Likewise if it's the undercrossing strand, we immediately reach the 1-handle to the left of the crossing, so $a$ appears in the differential in a word of the form $\mu_1 a c^0_{i\; p+1} \mu_2$ or  $\mu_1 c^0_{i\; p+1} a \mu_2$, $i<p+1$. 

We can make an equivalent argument for any of the $b_i$ following their overcrossing and undercrossing strands to the left or right, as we see that we do not meet any negative corners until we reach the 1-handle.

Since we have quotiented out the constant terms (any monomials of the form $e_1$) and the differential on products is generated by the Leibniz rule: the only way to obtain a monomial of smaller degree is if there is a constant term in one of the differentials, but no such term exists. Thus $a,b_1,\dots,b_{q-1}$ cannot appear as a degree 1 monomial in the differential of some product of generators. 
\end{proof}

\begin{eg}\label{eg:8_20DGA}
In this example, we fully compute the internal and external DGA of $\mathcal{D}(\Lambda)$ where $\mathcal{D}(\Lambda)$ is the Weinstein diagram from Example \ref{eg:8_20diagram} and Example \ref{eg:m10_140}. The Lagrangian resolution of $\mathcal{D}(\Lambda)$ is the diagram on the right of Figure \ref{fig:8_20LagRes}. We can then label the crossings which generate the DGA as in Figure \ref{fig:8_20DGA}. The generators are $t, a, a_1, b_1, b_2, b_3$, $c^0_{i\;j}$ for $1\leq i<j\leq 3$, and $c^1_{i\;j}$ for $1\leq i,j\leq 3$. 

\begin{figure}
	\begin{center}
		\includegraphics[width=12 cm]{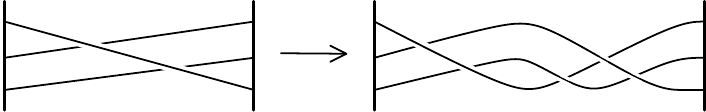};
		\caption[The Lagrangian resolution for $\mathcal{D}(m(8_{20}))$.]{The Lagrangian resolution for $\mathcal{D}(\Lambda)$, where $\Lambda$ is the $m(8_{20})$ knot.}
		\label{fig:8_20LagRes}
	\end{center}
\end{figure}

\begin{figure}
	\begin{center}
		\begin{tikzpicture}
			\node[inner sep=0] at (0,0) {\includegraphics[width=6 cm]{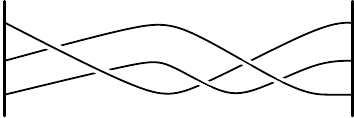}};
			\node at (-2,0.45){$a$};
			\node at (-1.2,0.05){$a_1$};
			\node at (0.5,-0.75){$b_1$};
			\node at (1.8,-0.75){$b_2$};
			\node at (1.1,0.3){$b_3$};
			\node at (3.2,-0.8){$1$};
			\node at (3.2,0){$2$};
			\node at (3.2,0.8){$3$};
			\node at (-3.2,0.8){$1$};
			\node at (-3.2,0){$2$};
			\node at (-3.2,-0.8){$3$};
			\node at (-0.2,-0.61){$\bullet$};
			\node at (-0.2,-1){$t$};
		\end{tikzpicture}
		\caption[$\mathcal{D}(m(8_20))$ labelled for a DGA computation.]{The external Reeb chords generating the DGA of $L$ in $\mathcal{D}(m(8_20))$ labelled.}
		\label{fig:8_20DGA}
	\end{center}
\end{figure}
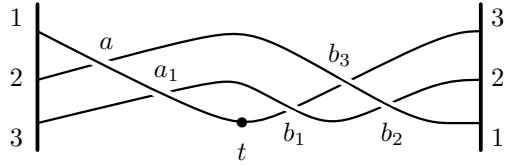

Then the gradings of the generators are:
\begin{align*}
	&|t|=|a|=|a_1|=|b_1|=|b_2|=|b_3|=0, &&|c^0_{i\;j}|= -1, && |c^1_{i\;j}|= 1
\end{align*}
We get following differentials for the external generators:
\begin{align*}
	\partial a &= c^0_{1\;2}\\
	\partial a_1  &= c^0_{1\;3}+ ac^0_{2\;3}\\
	\partial b_1  &= c^0_{2\;3}\\
	\partial b_2  &= c^0_{1\;2}+c^0_{2\;3}\\
	\partial b_3  &= c^0_{1\;3}+b_2c^0_{2\;3}+ c^0_{2\;3}b_1,
\end{align*}
along with differentials for the internal generators as follows:
\begin{align*}
	\partial c^0_{1\;2} &= 0\\
	\partial c^0_{1\;3} &= c^0_{1\;2}c^0_{2\;3}\\
	\partial c^0_{2\;3} &= 0\\
	\partial c^1_{1\;1} &=e_1+c^0_{1\;2}c^1_{2\;1}+c^0_{1\;3}c^1_{3\;1}\\
	\partial c^1_{1\;2} &=c^0_{1\;2}c^1_{2\;2}+c^0_{1\;3}c^1_{3\;2}+c^1_{1\;1}c^0_{1\;2}\\
	\partial c^1_{1\;3} &=c^0_{1\;2}c^1_{2\;3}+c^0_{1\;3}c^1_{3\;3}+c^1_{1\;1}c^0_{1\;3}+c^1_{1\;2}c^0_{2\;3}\\
	\partial c^1_{2\;1} &= c^0_{2\;3}\\
	\partial c^1_{2\;2} &=e_1+c^0_{2\;3}c^1_{3\;2}+c^1_{2\;1}c^0_{1\;2}\\
	\partial c^1_{2\;3} &=c^0_{2\;3}\\
	\partial c^1_{3\;1} &=0\\
	\partial c^1_{3\;2} &=c^1_{3\;1}c^0_{1\;2}\\
	\partial c^1_{3\;3} &=e_1+c^1_{3\;1}c^0_{1\;3}+c^1_{3\;2}c^0_{2\;3}.
\end{align*}
\end{eg}

\section{Nonvanishing symplectic homology}\label{sec:nonvanishing}

In this section, we use the Legendrian contact homology DGA to understand symplectic homology of the Weinstein manifold given by handle attachments according to $\mathcal{D}(\Lambda)$.

Symplectic homology and cohomology are very useful invariants of exact symplectic manifolds with contact type boundaries, introduced by Viterbo in \cite{Vit99}. Symplectic homology can be used to prove the existence of closed Hamiltonian orbits and Reeb chords, and the wrapped Fukaya category \cite{FukSei08} is built using the wrapped Floer cohomology, which are modules over the symplectic cohomology ring for non-compact symplectic manifolds. See \cite{Sei06} for a survey on symplectic homology. Work of Bourgeois and Oancea first related symplectic homology to linearized contact homology \cite{BouOan09} and a way to compute symplectic homology via the Legendrian contact homology DGA was established by Bourgeois, Ekholm, and Eliashberg in \cite{BouEkh12,Ekh19}. In this section, we summarize and apply the results of \cite{BouEkh12}. We begin with Corollary 5.7 of \cite{BouEkh12} which states:
\begin{thm}\cite{BouEkh12}\label{thm:BEE}
$$S\mathbb{H}(X) = L\mathbb{H}^{Ho}(L)$$
where $L\mathbb{H}^{Ho}(L)$ is the homology of the Hochschild complex associated to the Legendrian contact homology differential graded algebra of $L$ over $\QQ$.
\end{thm}

Thus, in order to compute the symplectic homology $S\mathbb{H}(X)$, we first compute 
$$LH^{Ho+}(L)\defeq \widecheck{LHO}^+(L)\oplus \widehat{LHO}^+(L),$$ 
defined as follows. We consider the DGA $\mathscr{A}_L$ defined in the previous section and generated by cyclically composable monomials of Reeb chords. Let $LHO(L)=\mathscr{A}_L$. Let 
$$LHO^+(L)\defeq LHO(L)/\langle e_1\rangle$$
be the subalgebra of $LHO(L)$ generated by non-trivial cyclically composable monomials of Reeb chords. Let 
$$\widecheck{LHO}^+(L)\defeq LHO^+(L)$$
and let 
$$\widehat{LHO}^+(L)\defeq LHO^+(L)[1],$$
that is $LHO^+(L)$ with grading shifted up by 1. Now, given a monomial $w=c_1\dots c_l\in LHO^+(L)$, we denote the corresponding elements in $\widecheck{LHO}^+(L)$ and $\widehat{LHO}^+(L)$ as $\check{w}=\check{c_1}\dots c_l$ and $\hat{w}=\hat{c_1}\dots c_l$, respectively. The hat or check may mark a variable in the monomial which is not the first one, in which case the monomial is the word obtained by the graded cyclic permutation which puts the marked letter in the
first position. Let $S:LHO^+\rightarrow \widehat{LHO}^+$ denote the linear operator defined by the formula:
$$S (c_1\dots c_l)\defeq \hat{c_1}c_2\dots c_l + (-1)^{|c_1|}c_1\hat{c_2}\dots c_l + \dots + (-1)^{|c_1 \dots c_l|}c_1 c_2\dots \hat{c_l}.$$
Then the differential $d_{Ho+}:LH^{Ho+}\rightarrow LH^{Ho+}$ is given by 
$$d_{Ho+} = \begin{pmatrix}
\check {d}_{LHO^+} & d_{M\;Ho+}\\
		0		 & \hat{d}_{LHO^+}
\end{pmatrix}.$$
The maps in the matrix on generators are as follows:
\begin{enumerate}
	\item If $w\in LHO^+(L)$ is a monomial, then 
	$$\check{d}_{LHO^+}(\check{w})\defeq \sum_{j=1}^r \check{v}_j,$$
	where $d_{LHO^+}(w)\defeq \partial_{\mathscr{A}_L} (w) =\sum_{j=1}^r v_j$ for monomials $v_j$.
	\item If $c$ is a chord and $w$ is a monomial such that $cw\in LHO^+(L)$, then
	$$\hat{d}_{LHO^+}(\hat{c} w) = S(d_{LHO^+}(c))w +(-1)^{|c|+1}\hat{c}(d_{LHO^+}(w)).$$
	\item If $w=c_1\dots c_l\in LHO^+(\Lambda)$, then 
	$$d_{M\;Ho+}(\hat{w}) \defeq \check{c_1}\dots c_l - c_1 \dots \check{c_l}.$$ 
\end{enumerate}

\begin{remark}\label{remark:dho}
Note that $d_{M\;Ho+}$ is zero on linear monomials.
\end{remark}

Now we can define $LH^{Ho}(L)\defeq LH^{Ho+}(L)\oplus C(L)$ where $C(L)$ is the vector space generated by a single element $\tau_1$ of grading 0. Then the differential $d_{Ho}:LH^{Ho}(L)\rightarrow LH^{Ho}(L)$ is defined as:
$$d_{Ho} = \begin{pmatrix}
d_{Ho+} & 0\\
\delta_{Ho} & 0
\end{pmatrix}.$$
For any chord $c$, we define $\delta_{Ho}(\check{c}) \defeq n_c\tau_1$ where $n_c$ is the count of the zero-dimensional moduli space of holomorphic disks asymptotic to $\infty$ at $c$. Then we have the following:
\begin{prop}\cite{BouEkh12}
$d_{Ho}^2 = 0$ and the homology 
$$L\mathbb{H}^{Ho}(L) = H_*(LH^{Ho}(L), d_{Ho})$$ is independent of choices and is a Legendrian isotopy invariant of $L$.
\end{prop}

We apply Theorem \ref{thm:BEE} to the diagrams $\mathcal{D}(\Lambda)$ obtained in Section \ref{sec:drawing_diagrams} to obtain the following theorem:

\newtheorem*{thm:nonvanishingSH}{Theorem \ref{thm:nonvanishingSH}}
\begin{thm:nonvanishingSH}
	Let $\Lambda\neq U$ be a Legendrian knot which is smoothly the closure of a quasipositive 3-braid of algebraic length 2. Let $\Lambda'$ be a positive transverse push off of $\Lambda$. Then there is a filling of $\Sigma_2(\Lambda')$, the double cover of $S^3$ branched over $\Lambda'$, which has nonvanishing symplectic homology.
\end{thm:nonvanishingSH}

\begin{proof}
Let $X_\Lambda$ be the filling of $\Sigma_2(\Lambda)$ given by the Weinstein handle decomposition depicted in $\mathcal{D}(\Lambda)$. We will show that 
$$S\mathbb{H}(X_\Lambda) = L\mathbb{H}^{Ho}(L)$$ 
is nonzero by finding a $(\check{w}, \hat{v}, a_i\tau_i)\in \widecheck{LHO}^+(L)\oplus \widehat{LHO}^+(L)\oplus C(L) = LH^{Ho}(\Lambda)$ such that $d_{Ho}((\check{w}, \hat{v}, a_i\tau_i)) = 0$ but $(\check{w}, \hat{v}, a_i\tau_i)\notin \text{Im}(d_{Ho})$.

First, we consider the Legendrian contact homology DGA $\mathscr{A}_L=LHO(L)$ of the link $L$ in the Weinstein diagram $\mathcal{D}(\Lambda)$. Recall that $\mathcal{D}(\Lambda)$ is constructed via surgery on a $(p,q)$ and a $(1,0)$ curve in the torus with one boundary component, and consists of the attaching curves of a single 2-handle and a single 1-handle. Since $\Lambda\neq U$, by Lemma \ref{lemma:qgreaterthanp}, we choose $p$ and $q$ satisfing $0<p<q$.

The attaching sphere of the 2-handle winds around the 1-handle $q$ times.  

As in subsection \ref{subsec:computeDGA}, we will write $a_i$ (or $a$) to denote generators coming from crossings present in the front diagram, $b_i$ to denote generators coming from crossings formed by the Lagrangian resolution, and $c^0_{i\;j}, c^1_{i\;j}$ to denote generators coming from the internal Reeb chords within the 1-handle. In particular, consider the generators $a$ and $b_i$, $1\leq i\leq q-1$ as labelled in Figure \ref{fig:GenDGA}. The differentials for these generators are computed explicitly in the previous subsection, see (\textasteriskcentered).

Note that every $c^0_{i\;i+1}$ appears in this collection of differentials twice, with the marked point giving an extra $t$ coefficient only to the term $c^0_{q-p+1\;q-p+2}$ in $\partial b_{p-1}$ when $p>1$.  
Note also that since $(p,q)=1$, for any $i\neq j$, $\partial b_i \neq \partial b_j$.  Let
$$c\defeq ta +\sum_{i=1}^{q-1} \epsilon_i t^{s_i} b_i$$ 
where we choose $s_i \in\{0,1\}$ and $\epsilon_i\in\{-1,1\}$ using the following procedure:
\begin{enumerate}
	\item Let $c_0\defeq a +\sum_{i=1}^{q-1} b_i$. Rearrange the terms in sum $c_0=a +\sum_{i=1}^{q-1} b_i$ in terms of the differentials of the generators so that matching terms are adjacent and relabel the $b_i$ in this order with a new index $b_j'$: 
	\begin{align*}
		\partial(c_0) =& (c^0_{p\;p+1}) + (c^0_{p\;p+1} + c^0_{i_0\;j_0}) + (c^0_{i_0\;j_0}+c^0_{i_1\;j_1}) + \dots\\
		&+ (tc^0_{q-p+1\;q-p+2}+c^0_{1\;2}) + (c^0_{1\;2}+c^0_{p+1\;p+2})+\dots+(c^0_{q-p\;q-p+1})\\
		=&\partial a + \partial b'_{1} + \dots +\partial b'_{q-1}\\
		=&\partial a + \sum_{j=1}^{q-1}\partial b_{j}'
	\end{align*}
	\item for $i \in\{1,\dots q-1\}$, if $b_i=b'_{j}$ and $j<p-1$, let $s_i=1$. Otherwise, let $s_i=0$. If $p=1$, $s_i=0$ for all $i$.
	\item Let $\epsilon_i= (-1)^j$.
\end{enumerate}
Then $\partial_{\mathscr{A}} (c)=0$.

Consider the element $(\check{c},\hat{c},\tau_1)\in LH^{Ho}(L)$. Then $\partial_{\mathscr{A}} (c)=d_{LHO}(c)=0$, so $d_{LHO^+}(c)=0$. Thus we obtain the following:
$$\check{d}_{LHO^+}(\check{c})= 0,$$
$$\hat{d}_{LHO^+}(\check{c}) = 0,$$
$$d_{M\;Ho+}(\hat{c}) = 0$$
by Remark \ref{remark:dho}. Finally, $\delta_{Ho}$ counts the zero-dimensional moduli space of holomorphic disks asymptotic to Reeb chords at $\infty$. These are disks with boundary consisting of a smooth curve along the front diagram that does not pass through a negative corner. Since the curves in $\mathcal{D}(\Lambda)$ pass monotonically left to right, any such boundary would necessarily have a negative corner at the 1-handle. Thus,
$$\delta_{Ho}(c) = 0.$$
Thus we conclude that $d_{Ho} (\check{c},\hat{c},\tau_1) = 0$. 

To see that $(\check{c},\hat{c},\tau_1)\notin \text{Im}(d_{Ho})$, note that in the image of $d_{LHO^+}$ is generated by the differentials of the generators of $\mathscr{A}_L/\langle e_1\rangle$. Thus we can apply Lemma \ref{lemma:nosolos}, and we know that the terms $a, b_1,\dots,b_{q-1}$ do not appear in any differentials of $\mathscr{A}_L/\langle e_1\rangle$ as degree 1 monomials. Thus $c \notin \text{Im}(d_{LHO^+})$.
\end{proof}

\begin{eg}\label{eg:8_20cycle}
Following the above proof, for the $m(8_{20})$ knot with differential algebra given by Example \ref{eg:8_20DGA}, the cycle $c$ is given by $a - b_2+b_1$. 
\end{eg}

\begin{eg}
\label{eg:10_155}

The $10_{155}$ knot \cite{knotatlas} is doubly slice, meaning that it appears as a cross section of an unknotted $S^2$ in $S^4$ \cite{LivMei15}, however from the main theorem, we can conclude that it cannot be Lagrangian doubly slice. In this example, we compute the DGA and find the cycle $c$ explicitly for this example.

The transverse $10_{155}$ is the closure of the braid $\sigma_1\sigma_2^{-2}\sigma_1\sigma_2^{-1}\sigma_1\sigma_2\sigma_1^{-1}\sigma_2^2$ \cite{SnapPy}.
By Corollary \ref{cor:WeinsteinLF}, its double cover has a Weinstein Lefschetz fibration $\pi:X\rightarrow D^2$ with generic fibre $F$ a punctured torus, and vanishing cycles given by $\alpha$ and $\tau_{\beta^{-2}\alpha\beta^{-1}}(\alpha)$ which are curves of slope $(1,0)$ and $(2,5)$ respectively. Following Theorem \ref{thm:diagram}, we obtain the surgery diagram $\mathcal{D}(10_{155})$, as in Figure \ref{fig:10155front}. 

\begin{figure}
	\begin{center}
		\includegraphics[width=5cm]{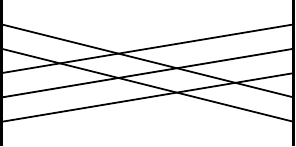}
		\caption{A front diagram of $\mathcal{D}(\Lambda)$ where $\Lambda$ is the knot $10_{155}$.}
		\label{fig:10155front}
	\end{center}
\end{figure}

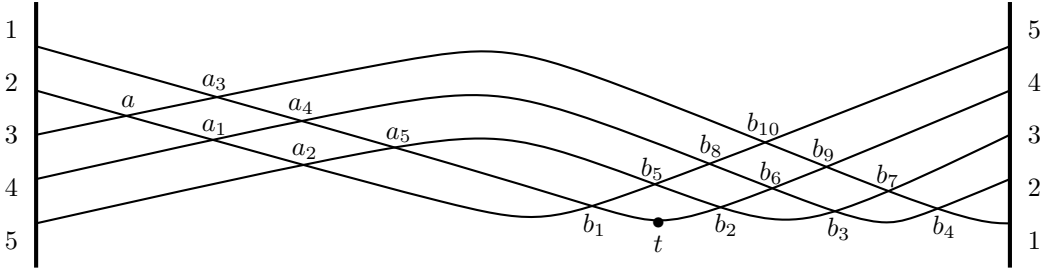
\begin{figure}
	\begin{center}
		\begin{tikzpicture}
			\node[inner sep=0] at (0,0) {\includegraphics[width=13cm]{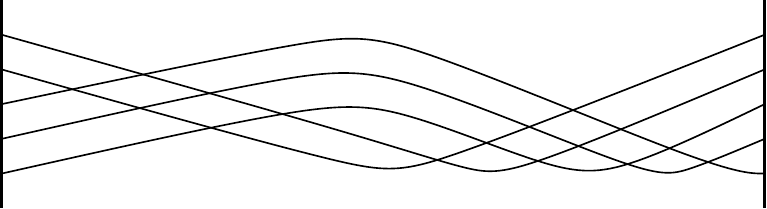}};
			\node at (-5.25,0.45){$a$};
			\node at (-4.1,0.1){$a_1$};
			\node at (-2.9,-0.22){$a_2$};
			\node at (-4.1,0.7){$a_3$};
			\node at (-2.95,0.38){$a_4$};
			\node at (-1.65,0){$a_5$};
			\node at (0.95,-1.2){$b_1$};
			\node at (2.7,-1.2){$b_2$};
			\node at (4.2,-1.25){$b_3$};
			\node at (5.6,-1.22){$b_4$};
			\node at (1.72,-0.45){$b_5$};
			\node at (3.3,-0.5){$b_6$};
			\node at (4.85,-0.55){$b_7$};
			\node at (2.5,-0.15){$b_8$};
			\node at (4,-0.2){$b_9$};
			\node at (3.2,0.15){$b_{10}$};
			\node at (6.8,-1.4){$1$};
			\node at (6.8,-0.7){$2$};
			\node at (6.8,0){$3$};
			\node at (6.8,0.7){$4$};
			\node at (6.8,1.4){$5$};
			\node at (-6.8,1.4){$1$};
			\node at (-6.8,0.7){$2$};
			\node at (-6.8,0){$3$};
			\node at (-6.8,-0.7){$4$};
			\node at (-6.8,-1.4){$5$};
			\node at (1.8,-1.17){$\bullet$};
			\node at (1.8,-1.45){$t$};
		\end{tikzpicture}
		\caption[The labelled Lagrangian resolution of $\mathcal{D}(\Lambda)$ where $\Lambda$ is the knot $10_{155}$.]{The Lagrangian resolution of $\mathcal{D}(\Lambda)$ where $\Lambda$ is the knot $10_{155}$. The Reeb chords generating the external DGA and the strands entering the 1-handles are labelled.}
		\label{fig:10155lag}
	\end{center}
\end{figure}

The Lagrangian resolution of $\mathcal{D}(\Lambda)$ is given by Figure \ref{fig:10155lag}. We label the crossings of this diagram with the conventions described in Section \ref{sec:DGA}. The generators are $t, a, a_1, \dots, a_5, b_1, \dots, b_{10}$, $c^0_{i\;j}$ for $1\leq i<j\leq 5$, and $c^1_{i\;j}$ for $1\leq i,j\leq 5$. The gradings are:
\begin{align*}
	&|t|=|a|=|a_i|=|b_i|=0 && |c^0_{i\;j}|= -1 &&|c^1_{i\;j}|= 1
\end{align*}

The differentials of generators of the internal DGA are:
\begin{align*}
	\partial(c^0_{i\;j}) &= \sum_{m=1}^5c^0_{i\;m} c^0_{m\;j}\\
	\partial(c^1_{i\;j}) &= \delta_{i\;j}+\sum_{m=1}^5 c^0_{i\;m} c^1_{m\;j}+\sum_{m=1}^n c^1_{i\;m} c^0_{m\;j}\\
\end{align*}
where $\delta_{i\;j}=e_i=e_j$ if $i=j$ and $0$ otherwise. The differentials of the generators of the external DGA are:
\begin{align*}
	\partial a   &=c^0_{2\;3}\\
	\partial a_1 &=c^0_{2\;4}+ac^0_{3\;4}\\
	\partial a_2 &=c^0_{2\;5}+ac^0_{3\;5}+a_1c^0_{4\;5}\\
	\partial a_3 &=c^0_{1\;2}a+c^0_{1\;3}\\
	\partial a_4 &=c^0_{1\;2}a_1+c^0_{1\;4}+a_3c^0_{3\;4}\\
	\partial a_5 &=c^0_{1\;2}a_2+c^0_{1\;5}+a_3c^0_{3\;5}+a_4c^0_{4\;5}\\
	\\
	\partial b_1 &=c^0_{1\;2}+tc^0_{4\;5}\\
	\partial b_2 &=c^0_{3\;4}\\
	\partial b_3 &=c^0_{2\;3}+c^0_{4\;5}\\
	\partial b_4 &=c^0_{1\;2}+c^0_{3\;4}\\
	\partial b_5 &=b_2c^0_{4\;5}+c^0_{3\;5}\\
	\partial b_6 &=b_3c^0_{3\;4}+c^0_{2\;4}+c^0_{4\;5}b_2\\
	\partial b_7 &=b_4c^0_{2\;3}+c^0_{1\;3}+c^0_{3\;5}+c^0_{3\;4}b_3\\
	\partial b_8 &=b_6c^0_{4\;5}+b_3c^0_{3\;5}+c^0_{2\;5}+c^0_{4\;5}b_5\\
	\partial b_9 &=b_7c^0_{3\;4}+b_4c^0_{2\;4}+c^0_{1\;4}+c^0_{3\;4}b_6+c^0_{3\;5}b_2\\
	\partial b_{10} &= b_9c^0_{4\;5}+b_7c^0_{3\;5}+b_4c^0_{2\;5}+c^0_{1\;5}+c^0_{3\;4}b_8+c^0_{3\;5}b_5\\
\end{align*}

Then the cycle in the symplectic homology of the filling depicted in $\mathcal{D}(\Lambda)$ as described in the proof of Theorem \ref{thm:nonvanishingSH} is given by $(\check{c},\hat{c},\tau_1)$, where 
$$c = ta_1 - tb_3 +b_1 - b_4 + b_2.$$
We see that $\partial c = 0$.

\end{eg}

\section{The Main Theorem}

We will use the filling with nonzero symplectic homology of the previous section along with a result of McLean \cite{McL09} to prove Theorem \ref{thm:main}. McLean's result is based on the work of Viterbo in \cite{Vit99}, specifically \emph{Viterbo functoriality} which says that a codimension 0 exact embedding of a symplectic manifold with boundary into another induces a \emph{transfer map} on the symplectic homologies between them. More precisely,

\begin{thm}\cite{Vit99}\label{thm:Viterbo}
Suppose $(W,d\lambda)$ is an exact symplectic manifold with a Liouville vector field which is transverse at the boundary. Suppose $W_0\hookrightarrow W$ is an embedding of a compact codimension 0 submanifold. Then there exists a natural homomorphism
$$S\mathbb{H}_*(W,d\lambda)\to S\mathbb{H}_*(W_0,d\lambda).$$
Moreover, this map, along with the natural map on relative singular homology $H_*(W,\partial W)\to H_*(W_0,\partial W_0)$ forms the following commutative diagram:
$$\begin{tikzcd}
{H_{*+n}(W,\partial W)} \arrow[rr] \arrow[d] &  & {H_{*+n}(W_0,\partial W_0)} \arrow[d] \\
{S\mathbb{H}_*(W,d\lambda)} \arrow[rr]                &  & {S\mathbb{H}_*(W_0,d\lambda)}                
\end{tikzcd}$$
\end{thm}

McLean checked that this transfer map was in fact a unital ring map and proved the following theorem (Cor 10.5 in \cite{McL09}):
\begin{thm}\cite{McL09}\label{thm:McLean} 
Let $X$ and $W$ be compact convex symplectic manifolds. Suppose $W$ is subcritical. Suppose $S\mathbb{H}(X)\neq0$. Then $X$ cannot be embedded in $W$ as an exact codimension 0 submanifold. In particular, if $H_1(X)=0$, then $X$ cannot be symplectically embedded into $W$.
\end{thm}

We prove the main theorem:

\newtheorem*{thm:main}{Theorem \ref{thm:main}}
\begin{thm:main}
	Let $U$ be the standard $tb=-1$ unknot. Let $\Lambda$ be a Legendrian knot satisfying $U\prec\Lambda\prec U$, and $\Lambda\neq U$. Then $\Lambda$ cannot be smoothly the closure of a 3-braid.
\end{thm:main}

\begin{proof}
We proceed by contradiction. Suppose $\Lambda\neq U$ is smoothly the closure of a 3-braid which satisfies $U\prec\Lambda\prec U$. $\Lambda$ must be quasipositive. $\Lambda$ has a positive transverse push off $\Lambda'$ of the same topological knot type. By Remark \ref{remark:quasipositive_transverse}, $\Lambda'$ is transversely isotopic to a quasipositive 3-braid. By Corollary \ref{cor:alg_length}, $\Lambda$ must have algebraic length 2. By Theorem \ref{thm:diagram}, $\mathcal{D}(\Lambda)$ is the Weinstein diagram of a filling of $\Sigma_2(\Lambda')$, call it $X_\Lambda$. By Theorem \ref{thm:nonvanishingSH}, $X_\Lambda$ has nonzero symplectic homology. 

Next, we'll show that $X_\Lambda$ embeds in $B^4$ as a codimension 0 exact submanfold. Recall the construction from the proof of Theorem \ref{thm:fillings_embed}. In particular, we will need the submanifold $V$ of $\Sigma_p(C')$, where $C'$ is the symplectic approximation of the Lagrangian concordance cylinder of $U\prec\Lambda\prec U$. Here we fix $p=2$. Then we have $\partial V = \Sigma_2(\Lambda')\cup S^3$. In the proof of Theorem \ref{thm:fillings_embed}, we showed that any 4-manifold constructed by gluing a filling of $\Sigma_2(\Lambda')$ to $V$ must embed in a blow-up of $B^4$.

Let $X$ be the branched double cover of $B^4$ branched over the symplectic disk bounding $\Lambda'$. Then the manifold we get by gluing $X$ to $V$ along $\Sigma_2(\Lambda')$ is $B^4$. Then we obtain the following Mayer-Vietoris sequence:
$$
\begin{tikzcd}
H_3(B^4) \arrow[r] & H_2(\Sigma_2(\Lambda')) \arrow[r] & H_2(X)\oplus H_2(V) \arrow[r] & H_2(B^4)
\end{tikzcd}
$$ 

$\Sigma_2(\Lambda')$ is a rational homology sphere, thus $H_2(\Sigma_2(\Lambda'),\QQ)=0$. Then since $H_3(B^4,\QQ)=H_2(X,\QQ)= H_2(B^4)=0$, we have that 
$$H_2(V,\QQ)=0.$$

Now instead take $W = X_\Lambda\cup V$ to be the manifold obtained by gluing $X_\Lambda$ to $V$.  We obtain the following Mayer-Vietoris sequence:
$$
\begin{tikzcd}
H_2(\Sigma_2(\Lambda')) \arrow[r] & H_2(X_\Lambda)\oplus H_2(V) \arrow[r] & H_2(W) \arrow[r] & H_1(\Sigma_2(\Lambda'))
\end{tikzcd}
$$

From $\mathcal{D}(\Lambda)$, we see that $X_\Lambda$ is constructed via attaching a single 1-handle and a single 2-handle to a 0-handle. The 2-handle is attached along a curve which runs along the 1-handle nontrivially. Thus $H_2(X_\Lambda,\QQ)=0$. And since $H_2(\Sigma_2(\Lambda'),\QQ)=H_2(V,\QQ)= H_1(\Sigma_2(\Lambda'),\QQ)=0$, 
$$H_2(W,\QQ)=0.$$ 
Thus $W$ must be minimal, so $W$ is $B^4$, which is subcritical. Since $X_\Lambda$ has a Weinstein structure, it is exact and convex.

This contradicts Theorem \ref{thm:McLean}. Thus if $\Lambda$ satisfies $U\prec\Lambda\prec U$, $\Lambda\neq U$, then $\Lambda$ is not smoothly the closure of a 3-braid. 
\end{proof}

Additionally the following result about contact embeddings follows from the proof of Theorem \ref{thm:main}. The double covers of $S^3$ branched over a quasipositive transverse knot $\Lambda'$ which is the closure of a 3-braid of algebraic length 2  form an infinite family of contact manifolds which are rational homology spheres but do not embed in $\RR^4$ as contact type hypersurfaces, motivated by the work of Mark and Tosun \cite{MarTos20}.

\newtheorem*{cor:not_hyper}{Corollary \ref{cor:not_hyper}}
\begin{cor:not_hyper}
	Let $\Sigma_2(\Lambda')$ be the double cover of $S^3$ branched over a quasipositive transverse knot which is the closure of a 3-braid of algebraic length 2. Suppose $\Lambda'$ is not the unknot. Then $\Sigma_2(\Lambda')$ does not embed as a contact type hypersurface in $\RR^4$.
\end{cor:not_hyper}

\begin{proof}
Suppose $\Sigma_2(\Lambda')$ be the double cover of $S^3$ branched over a quasipositive transverse knot $\Lambda'$ which is the closure of a 3-braid of algebraic length 2. Suppose $\Lambda'$ is not the unknot. Suppose $\Sigma_2(\Lambda')$ embeds as a contact type hypersurface in $\RR^4$. 
Then it bounds a codimension-0 symplectic submanifold of $\RR^4$, call it $X$. 

We now follow the same arguments as in the proof of Theorem \ref{thm:main}, to reach a contradiction. We have a Mayer-Vietoris sequence:
$$
\begin{tikzcd}
H_3(\RR^4) \arrow[r] & H_2(\Sigma_2(\Lambda')) \arrow[r] & H_2(X)\oplus H_2(\RR^4\setminus X) \arrow[r] & H_2(\RR^4)
\end{tikzcd}
$$ 
So $H_2(X)=0$. Consider $(\RR^4\setminus X)$  which has boundary $\Sigma_2(\Lambda')$. We can glue in the filling $X_\Lambda$ of $\Sigma_2(\Lambda')$ corresponding to the diagram $\mathcal{D}(\Lambda)$ of Theorem \ref{thm:diagram} along the boundary $\Sigma_2(\Lambda')$. Let $W\defeq (\RR^4\setminus X)\cup X_\Lambda$. By p. 311 of \cite{Gro85}, $W = \RR^4\#m\overline{\cptwo}$ for some $m\geq 0$. We have another Mayer-Vietoris sequence:
$$
\begin{tikzcd}
H_2(\Sigma_p(\Lambda')) \arrow[r] & H_2(X_\Lambda)\oplus H_2(B^4\setminus X) \arrow[r] & H_2(W) \arrow[r] & H_1(\Sigma_p(\Lambda')).
\end{tikzcd}
$$
Thus, $H_2(W)=0$ and we know that $W$ is $\RR^4$. Thus $X_\Lambda$ embeds as an exact codimension 0 submanifold of $W=\RR^4$. 

By Theorem \ref{thm:nonvanishingSH}, $X_\Lambda$ has nonzero symplectic homology. Thus by Theorem \ref{thm:McLean}, $X_\Lambda$ cannot embed in $\RR^4$, a contradiction.
\end{proof}

\begin{remark}\label{remark:only_some_ct_str}
The family of contact manifolds described by this corollary corresponds precisely to the Stein rational homology balls in the boundary of the Weinstein diagrams of Figure \ref{fig:d_Lambdaagain}. As contact surgery diagrams, we depict these manifolds in Figure \ref{fig:QHS_family}. We note that the methods in this paper exclude specifically the contact structures determined by the double branched covers from embedding as contact type hypersurfaces, and not necessarily all tight contact structures on these manifolds. 
\end{remark}

\begin{figure}
	\begin{center}
		\includegraphics[width=11 cm]{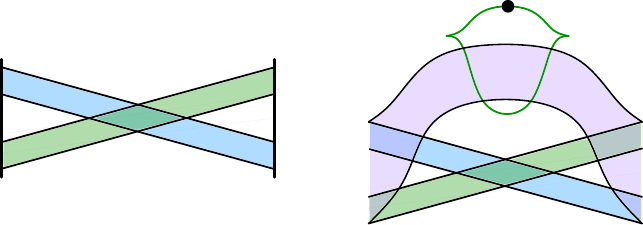}
		\caption{The Weinstein diagram of $X_{\Lambda}$ (left) and the corresponding surgery diagram of $\Sigma_2(\Lambda')$ (right). Here the shaded blue region represents $p$ parallel strands, green represents $q-p$ parallel strands, and purple represents $q$ parallel strands, where $0<p<q$ and $(p,q)=1$.}
		\label{fig:QHS_family}
	\end{center}
\end{figure}

To see that these $\Sigma_2(\Lambda')$ of Corollary \ref{cor:not_hyper} are indeed an infinite family, we distinguish infinitely many of them using their first homology, which we can compute from their contact surgery diagrams.

\begin{lemma}
There are infinitely many non homeomorphic manifolds which are double covers of $S^3$ branched over quasipositive transverse knots which are the closures of 3-braids of algebraic length 2.
\end{lemma}

\begin{figure}
	\begin{center}
		\includegraphics[width=11 cm]{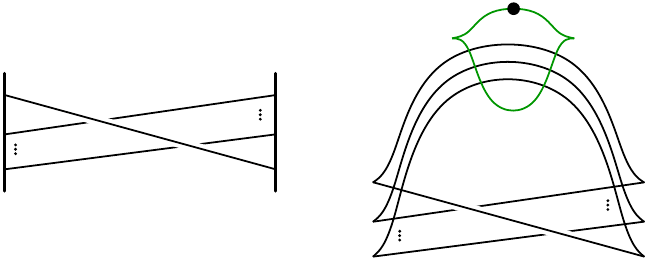}
		\caption[Diagrams of $X_{\Lambda_k}$ and $\Sigma_2(\Lambda'_k)$]{The Weinstein diagram of $X_{\Lambda_k}$ with $k$ strands entering the 1-handle (left) and the corresponding surgery diagram of $\Sigma_2(\Lambda'_k)$ (right).}
		\label{fig:q_squared}
	\end{center}
\end{figure}

\begin{proof}
Consider the 3-braids $\beta_k \defeq \sigma_1 \sigma_2^{-k}\sigma_1\sigma_2^{k}$ for $k\in\ZZ, k\geq 2$. Let $\Lambda'_k$ be a transverse knot which is the closure of $\beta_k$. By Proposition \ref{prop:OBD}, there is an open book decomposition of the double cover of $S^3$ branched over $\Lambda'_k$, $\Sigma_2(\Lambda'_k)$, with pages consisting of tori with one boundary component and monodromy $\phi=\tau_\alpha\tau_\gamma$ where $\alpha$ is a $(1,0)$ curve and $\gamma$ is a curve with slope $(1,k)$.

Then by Theorem \ref{thm:diagram}, $\Sigma_2(\Lambda'_k)$ has a filling $X_{\Lambda_k}$ with surgery diagram as in Figure \ref{fig:q_squared}. We can replace the 1-handle in this diagram with a 0-framed surgery on an unknot to obtain a surgery diagram of $\Sigma_2(\Lambda'_k)$, as in Figure \ref{fig:q_squared}. We call the unknot $U$ and the other knot $L$. We see that $U$ has linking number $k$ with $L$, so their linking matrix is 
$$\begin{bmatrix}
0 & k \\
k & tb(L)-1 
\end{bmatrix}$$
which has determinant $-k^2$, thus 
$$|H_1 (\Sigma_2(\Lambda'_k))|=k^2.$$
Thus, $\Sigma_2(\Lambda'_k)$ is not homeomorphic to $\Sigma_2(\Lambda'_j)$ for $j\neq k$.
\end{proof}

\begin{remark}
We note that the orientation reversed version of the family of manifolds depicted in Figure \ref{fig:q_squared} has been previously studied using different methods by Tosun \cite{Tos22} for $k\neq 2$ where it is shown that these manifolds do not embed smoothly in $\RR^4$, and by Nemirovski and Siegel \cite{NemSie16} for $k=2$ where it is shown that this manifold cannot be a contact type hypersurface. 
\end{remark}